\documentclass[11pt,a4paper]{article}

\usepackage{amsmath}
\usepackage{amsfonts}
\usepackage{enumerate}
\usepackage{amssymb}
\usepackage{accents}
\usepackage{amsthm}
\usepackage{parskip}
\usepackage{graphicx}
\usepackage{caption}
\usepackage{sidecap}
\usepackage{subcaption}
\usepackage{float}
\usepackage{epstopdf}
\usepackage{verbatim}
\usepackage[T1]{fontenc}
\usepackage[utf8]{inputenc}

\usepackage{mathrsfs}
\usepackage[pdftex,colorlinks,citecolor=cyan,linkcolor=magenta,urlcolor=blue,breaklinks=true]{hyperref}
\usepackage{cite}

\usepackage{booktabs}
\newtheoremstyle{examplestyletwo}  
  {5mm}       
  {3mm}       
  {\itshape}   
  {}        
  {\bfseries}  
  {}   
  {3mm}       
  {}           
  
\theoremstyle{examplestyletwo}
\newtheorem{theorem}{Theorem}[section]
\newtheorem{lemma}[theorem]{Lemma}

\newtheorem{corollary}[theorem]{Corollary}

\newcommand{\by}{{\bf y}}
\newcommand{\bz}{{\bf z}}

\newcommand{\bx}{{\boldsymbol \xi}}
\newcommand{\bpsi}{{\boldsymbol \psi}}
\newcommand{\bmu}{{\boldsymbol \mu}}
\newcommand{\bp}{{\bf p}}
\newcommand{\bd}{{\bf b_0}}

\newtheoremstyle{examplestyleone}  
  {3mm}       
  {10mm}       
  {\normalfont}   
  {}        
  {\bfseries}  
  {\quad}   
  {1mm}       
  {}           

\theoremstyle{examplestyleone}
\newtheorem{definition}{Definition}
\newtheorem{remark}{Remark}

\newtheorem{example}{Example}

\usepackage{fancyhdr}
\fancypagestyle{plain}{ %
  \fancyhf{} 
  
}

\everymath{\displaystyle}
\allowdisplaybreaks
\numberwithin{equation}{section}
\graphicspath{{plots//}}

\title{Global minima for optimal control of the obstacle problem} 
\author{Ahmad Ali\footnote{Schwerpunkt Optimierung und Approximation, 
Universit\"at Hamburg, Bundesstra{\ss}e 55, 20146 Hamburg, Germany.}, Klaus Deckelnick\footnote{Institut f\"ur Analysis und Numerik,
Otto--von--Guericke--Universit\"at Magdeburg, Universit\"atsplatz 2,
39106 Magdeburg, Germany} \& Michael Hinze\footnote{Schwerpunkt Optimierung und Approximation, 
Universit\"at Hamburg, Bundesstra{\ss}e 55, 20146 Hamburg, Germany.}}

\date{}

\begin{document}

\maketitle

\begin{abstract}
An optimal control problem subject to an elliptic obstacle problem is studied. We obtain a numerical approximation of this problem by discretising the PDE obtained via a Moreau--Yosida type penalisation. For the
resulting discrete control problem we provide a condition that allows to decide whether a solution of the necessary first order conditions is a global minimum. In addition we show that the corresponding result can
be transferred to the limit problem provided that the above  condition holds
uniformly in the penalisation and discretisation parameters. Numerical examples with unique global solutions are presented.
\end{abstract}


\begin{center}
{\bf \LARGE  }
\end{center}  


\section{Introduction}
In this paper we are concerned with the following distributed optimal control problem for the elliptic obstacle problem
\[
(\mathbb{P}) \quad
 \min_{u \in L^2(\Omega)} J(u)=\frac{1}{2} \Vert y-y_0 \Vert_{L^2(\Omega)}^2  + \dfrac{\alpha}{2} \Vert u \Vert_{L^2(\Omega)} ^2 
\]
subject to 
\begin{equation} \label{obstacle}
y \in K, \quad \int_{\Omega} \nabla y \cdot \nabla ( \phi - y) dx \geq \int_{\Omega} (f+u) (\phi -y) dx \qquad \forall \phi \in K,
\end{equation}
where
\begin{displaymath}
K= \lbrace \phi \in H^1_0(\Omega) \, | \, \phi(x) \geq \psi(x) \mbox{ a.e. in } \Omega \rbrace.
\end{displaymath}
Furthermore, $\Omega \subset \mathbb{R}^d \, (d=2,3)$ is a bounded polyhedral domain,   $\psi \in H^1_0(\Omega)\cap H^2(\Omega)$ is the given obstacle, $y_0, f \in L^2(\Omega)$ and $\alpha>0$. \\
It is well--known that for a given function $u \in L^2(\Omega)$ the variational inequality (\ref{obstacle}) has a unique solution $y \in K$ and using standard arguments (cf. \cite[Theorem 2.1]{MP84})
one obtains the existence of a solution of $(\mathbb{P})$. However, a major issue in the analysis and numerical approximation of $(\mathbb{P})$ is the fact
that the  mapping $u \mapsto y$ is in general not G\^{a}teaux differentiable, so that the derivation of necessary first order optimality conditions becomes a difficult task. A common
approach in order to handle this difficulty consists in approximating  (\ref{obstacle}) by a sequence of penalised or regularised problems and then to pass to the limit, see
e.g. \cite{Ba84}, \cite{MP84}, \cite{IK00}, \cite{H01}, \cite{HK11}, \cite{KW12a}, \cite{SW13} and \cite{MRW15}. In our work we  employ a Moreau--Yosida type penalisation of the obstacle problem
resulting in the following optimal control problem depending on the parameter $\gamma \gg 1$:
\[
(\mathbb{P}^{\gamma}) \quad
 \min_{u \in L^2(\Omega)} J^{\gamma}(u)=\frac{1}{2} \Vert y-y_0 \Vert_{L^2(\Omega)}^2  + \dfrac{\alpha}{2} \Vert u \Vert_{L^2(\Omega)} ^2 
\]
subject to
\begin{equation} \label{equation: 1}
\int_{\Omega} \nabla y \cdot \nabla \phi \, dx + \gamma^3 \int_{\Omega} [(y-\psi)^-]^3 \, \phi \, dx = \int_{\Omega} (f+u) \phi \, dx \qquad \forall \phi \in H^1_0(\Omega).
\end{equation}
Here, $a^-=\min(a,0)$. 
Existence of a solution of $(\mathbb{P}^{\gamma})$ and a detailed convergence analysis as $\gamma \rightarrow \infty$ can be found in \cite[Section 3]{MRW15}.  For numerical purposes  we shall discretise the PDE (\ref{equation: 1})
with the help of continuous, piecewise linear finite elements
giving rise to a discrete optimisation problem, whose solutions satisfy standard first order optimality conditions.
Due to the lack of convexity of the underlying  problem it is however not clear whether
a computed discrete stationary point is actually a global minimum. Our first main result of this paper establishes for a fixed penalisation parameter and a discrete stationary
solution a condition that ensures global optimality, see Theorem \ref{main2}. This condition has the form of an inequality involving the state and the adjoint variable as well as the obstacle. Furthermore, the
minimum is unique in case that the inequality is strict. A similar kind of result
for control problems subject to a class of semilinear elliptic PDEs has been obtained in \cite{Hinze1,Hinze2}, but the form of the condition presented here is adapted to the obstacle problem and entirely different from the one proposed
in \cite{Hinze1,Hinze2}. In Section 3 we consider a sequence of approximate control problems $(\mathbb{P}^{\gamma_h}_h)$ with $h \rightarrow 0$ and $\gamma_h \rightarrow \infty$ as $h \rightarrow 0$, where
$h$ denotes the grid size. As our second main result we shall prove that a corresponding sequence of discrete stationary points which are
uniformly bounded in a suitable sense has a subsequence that converges to a limit satisfying a system of first order optimality conditions, see Theorem \ref{stationarylimit}. It turns out that a solution of  this system is
strongly stationary in the sense of \cite{MP84} and that  we obtain a continuous analogue of the condition mentioned above guaranteeing global optimality of a stationary point, see Theorem \ref{statprop}. The only other sufficient
condition for optimality which we are aware of can be found in \cite[Theorem 5.4]{M76}, where global optimality is derived from the condition that $y_0 \leq \psi$, see also \cite[Section 5.2]{IK00}. A sufficient
second order optimality condition giving local optimality is derived in \cite{KW12b}.
In Section \ref{Sec4} we briefly outline how to apply our theory to a direct discretisation of (\ref{obstacle}).  Finally, several numerical examples with unique global solutions are presented in Section \ref{secn}.

\section{Discretisation of $(\mathbb{P})$ }
Let $\mathcal{T}_h$ be an admissible triangulation of $\Omega$ so that 
\[
\bar\Omega= \bigcup_{T \in \mathcal{T}_h} \bar T. 
\]
We denote by $x_1,\ldots,x_n$ the interior and by $x_{n+1},\ldots,x_{n+m}$ the boundary vertices of $\mathcal T_h$. Next, let
\[
X_{h}:= \{ v_h \in C^0(\bar\Omega) : v_h \mbox{ is a linear polynomial on each }  T \in \mathcal{T}_h    \}
\]
be the space of linear finite elements  as well as $X_{h0}:= X_h \cap H^1_0(\Omega)$.  The standard nodal basis functions are defined by
$\phi_i \in X_h$ with $\phi_i(x_j)=\delta_{ij}, i,j=1,\ldots,n+m$. In particular,  $\lbrace \phi_1,\ldots,\phi_n \rbrace$ is a basis of $X_{h0}$.
We shall make use of  the Lagrange interpolation operator 
\begin{displaymath}
I_h:C^0(\bar{\Omega}) \rightarrow X_h \quad I_h y:= \sum_{j=1}^{n+m} y(x_j) \phi_i. 
\end{displaymath}

\noindent
Let us approximate \eqref{equation: 1} as follows: for a given $u \in L^2(\Omega)$, find $y_h \in  X_{h0}$ such that
\begin{equation}
\label{equation: 3}
 \int_\Omega \nabla y_h   \cdot  \nabla \phi_h \, dx + \gamma^3 \int_{\Omega} I_h \{ [(y_h-\psi)^-]^3 \,  \phi_h \} \, dx = \int_\Omega (f+ u) \phi_h \,dx \quad \forall \, \phi_h \in X_{h0}.
\end{equation}
\noindent
It is not difficult to show that \eqref{equation: 3} has a unique solution $y_h \in X_{h0}$. The variational discretization of Problem~$(\mathbb{P}^{\gamma})$ now reads:
\[
(\mathbb{P}^{\gamma}_h) \quad
\begin{array}{l}
 \min_{u \in L^2(\Omega)} J^{\gamma}_h(u):=\frac{1}{2} \Vert y_h-y_0 \Vert_{L^2(\Omega)}^2  + \dfrac{\alpha}{2} \Vert u \Vert_{L^2(\Omega)} ^2 \\
\mbox{ subject to } y_h \mbox{ solves (\ref{equation: 3}) }. 
\end{array}
\]

\noindent
Using standard arguments one obtains:

\begin{lemma} \label{exitmult}
Suppose that $u_h \in L^2(\Omega)$ is a local solution of $(\mathbb{P}^{\gamma}_h)$ with corresponding state $y_h \in X_{h0}$. Then there exists an adjoint state $p_h \in X_{h0}$ 
such that 
\begin{eqnarray}
\hspace{-5mm} \int_\Omega \nabla y_h   \cdot  \nabla \phi_h  dx + \gamma^3 \int_{\Omega} I_h \{ [(y_h-\psi)^-]^3 \,  \phi_h \}  dx & = & \int_\Omega (f+ u_h) \phi_h \,dx \quad \forall \, \phi_h \in X_{h0} \label{foc1} \\
\hspace{-5mm} \int_{\Omega} \nabla p_h \cdot \nabla \phi_h  dx  + 3 \gamma^3 \int_{\Omega} I_h \{ [(y_h - \psi)^-]^2 p_h \phi_h \}  dx  & = & \int_{\Omega} (y_h - y_0) \phi_h \, dx \quad \forall \, \phi_h \in X_{h0}   \label{foc2} \\
 \alpha u_h + p_h & = & 0 \quad \mbox{ in } \Omega. \label{foc3}
\end{eqnarray}
\end{lemma}

\noindent
Note that (\ref{foc3}) implicitly yields a discretisation of the control variable so that (\ref{foc1})--(\ref{foc3}) is a finite--dimensional system that
can be solved using classical nonlinear programming algorithms. However, due to the non-convexity of the problem, it is not clear whether a solution of (\ref{foc1})--(\ref{foc3})
is actually a global minimum of $(\mathbb{P}^{\gamma}_h)$. Our first main result provides a sufficient condition for a discrete stationary point that guarantees that this is the case.
In order to formulate the corresponding condition  we introduce $\lambda_1$ as the smallest eigenvalue of $-\Delta$ in $\Omega$ subject to homogeneous Dirichlet boundary conditions, i.e.
\begin{equation}  \label{lambda1}
\displaystyle \lambda_1 = \inf_{\phi \in H^1_0(\Omega) \setminus \lbrace 0 \rbrace} \frac{\int_{\Omega} | \nabla \phi |^2 \, dx}{\int_{\Omega} \phi^2 \, dx}.
\end{equation}
\noindent
In what follows we shall abbreviate $y_j=y_h(x_j), p_j=p_h(x_j)$ and $\psi_j=\psi(x_j),\, j=1,\ldots,n$.

\begin{theorem} \label{main2}
Suppose that $(u_h,y_h,p_h)$ is a solution of (\ref{foc1})--(\ref{foc3}), which satisfies
\begin{equation} \label{foc4}
p_k \geq 0 \quad \forall k \in \lbrace 1,\ldots,n \rbrace \mbox{ with } y_k - \psi_k  =  0 
\end{equation}
and define
\begin{equation}  \label{defeta}
\eta:= \min \bigl(\min_{y_j > \psi_j} \frac{p_j}{y_j - \psi_j},\min_{y_j < \psi_j} \frac{3 p_j}{\psi_j-y_j},0 \bigr).
\end{equation}
If 
\begin{equation} \label{optcondgammah}
| \eta | \leq \alpha \lambda_1 + \sqrt{ \alpha^2 \lambda_1^2 + \alpha},
\end{equation}
then $ u_h$ is a global minimum for Problem~$(\mathbb{P}^{\gamma}_h)$. If the inequality (\ref{optcondgammah}) is strict, then
$ u_h$ is the unique global minimum. 
\end{theorem}
\begin{proof}  Let $v \in L^2(\Omega)$ be arbitrary and denote by $\tilde{y}_h \in X_{h0}$ the solution of 
\begin{equation}
\label{equation: 4}
\int_\Omega \nabla \tilde y_h   \cdot  \nabla \phi_h  dx + \gamma^3 \int_{\Omega} I_h \{ [(\tilde y_h-\psi)^-]^3 \,  \phi_h \}  dx  =  \int_\Omega (f+ v) \phi_h \,dx \quad \forall \, \phi_h \in X_{h0}.
\end{equation}
A straightforward calculation shows that
\begin{eqnarray}
J^{\gamma}_h(v) - J^{\gamma}_h(u_h) & = & \frac{1}{2} \Vert \tilde y_h - y_h \Vert_{L^2(\Omega)}^2 + \frac{\alpha}{2} \Vert v - u_h \Vert_{L^2(\Omega)}^2 \nonumber  \\
& & + \int_{\Omega} (y_h - y_0)( \tilde y_h - y_h) \, dx  + \alpha \int_{\Omega} u_h (v-u_h) \, dx. \label{jghdif1}
\end{eqnarray}
We deduce from (\ref{foc2}), (\ref{foc1}) and (\ref{equation: 4}) that
\begin{eqnarray}
\lefteqn{ \hspace{-5mm}
\int_{\Omega} (y_h - y_0)( \tilde y_h - y_h) \, dx = \int_{\Omega} \nabla p_h \cdot \nabla ( \tilde y_h - y_h ) \, dx + 3 \gamma^3 \int_{\Omega} I_h \{ [(y_h - \psi)^-]^2 p_h (\tilde y_h - y_h) \} dx } \nonumber \\
& = & - \gamma^3 \int_{\Omega} I_h \{ [ (\tilde y_h - \psi)^-]^3 p_h \} dx + \gamma^3 \int_{\Omega} I_h \{ [ (y_h - \psi)^-]^3 p_h \} dx + \int_{\Omega} p_h (v-u_h)  \, dx \nonumber  \\
& & + 3 \gamma^3 \int_{\Omega} I_h \{ [(y_h - \psi)^-]^2 p_h (\tilde y_h - y_h) \} dx  \nonumber \\
& = & \gamma^3 \int_{\Omega} I_h \{ p_h r_h \} \, dx + \int_{\Omega} p_h (v-u_h)  \, dx, \label{rterm}
\end{eqnarray}
where
\begin{eqnarray*}
r_h  & = & - [(\tilde y_h - \psi)^-]^3 + [(y_h - \psi)^-]^3 + 3 [(y_h - \psi)^-]^2  (\tilde y_h - y_h) \\
& = &  -  [(\tilde y_h - \psi)^-]^3 -2   [(y_h - \psi)^-]^3 + 3  [ (y_h - \psi)^-]^2   (\tilde y_h - \psi).
\end{eqnarray*}
\noindent
Using Young's inequality we find that
\begin{eqnarray*}
[ (y_h - \psi)^-]^2 (\tilde y_h - \psi) & = & [ (y_h - \psi)^-]^2 \bigl( (\tilde y_h - \psi)^- + (\tilde y_h - \psi )^+ \bigr) \\
& \geq & [ (y_h - \psi)^-]^2 (\tilde y_h - \psi)^- \geq \frac{2}{3} [(y_h -\psi)^-]^3 + \frac{1}{3} [(\tilde y_h-\psi)^-]^3
\end{eqnarray*}
so that
\begin{equation}  \label{rjge0}
r_h \geq 0 \quad \mbox{ in } \bar \Omega.
\end{equation}
\noindent
Inserting (\ref{rterm}) into (\ref{jghdif1}) and applying (\ref{foc3}) we obtain
\begin{equation} \label{jgdif2}
J^{\gamma}_h(v) - J^{\gamma}_h(u_h)  =   \frac{1}{2} \Vert \tilde y_h - y_h \Vert_{L^2(\Omega)}^2  + \frac{\alpha}{2} \Vert v - u_h \Vert_{L^2(\Omega)}^2   + \gamma^3 \sum_{j=1}^n p_j r_j m_j,
\end{equation}
where we have abbreviated $r_j:= r_h(x_j) \geq 0$ and  $m_j=\int_{\Omega} \phi_j \, dx$. \\
\noindent
Let us decompose
\begin{displaymath}
\lbrace 1,\ldots,n \rbrace = \lbrace j \, | \, y_j > \psi_j \rbrace \cup \lbrace j \, | \, y_j = \psi_j \rbrace \cup \lbrace j \, | \, y_j < \psi_j \rbrace =: \mathcal N^+ \cup \mathcal N^0 \cup \mathcal N^-.
\end{displaymath}
(i) $j \in \mathcal N^+$: In this case  we have in view of (\ref{defeta}) that $p_j \geq \eta (y_j - \psi_j)$ and hence 
\begin{eqnarray}
 p_j r_j  & \geq &  \eta  (y_j - \psi_j) r_j  = -  \eta (y_j - \psi_j) [(\tilde{y}_j-\psi_j)^-]^3   \nonumber  \\
 & = & \eta (\tilde y_j - y_j) [(\tilde{y}_j-\psi_j)^-]^3 - \eta [(\tilde{y}_j-\psi_j)^-]^4 \label{case1}  \\
 & \geq & \eta (\tilde{y}_j - y_j) \{ [(\tilde{y_j}-\psi_j)^-]^3 - [(y_j-\psi_j)^-]^3 \} \nonumber
\end{eqnarray}
since $(y_j - \psi_j)^{-} =0, j \in \mathcal N^+$ and $\eta \leq 0$. \\
(ii) $j \in \mathcal N^-$: In this case  we have $p_j \geq -\frac{\eta}{3}(y_j - \psi_j)$ so that
\begin{eqnarray}
p_j r_j &  \geq & - \frac{\eta}{3}  (y_j - \psi_j) r_j \\
&=& \frac{\eta}{3}  (y_j - \psi_j) [(\tilde{y}_j-\psi_j)^-]^3 + \frac{2}{3} \eta [(y_j-\psi_j)^-]^4  - \eta  [(y_j-\psi_j)^-]^3 (\tilde{y}_j - \psi_j) \nonumber  \\
& = & \eta (\tilde{y}_j - y_j) \{ [(\tilde{y}_j-\psi_j)^-]^3 - [(y_j-\psi_j)^-]^3 \} m_j \nonumber \\
& & - \eta \{   [(\tilde{y}_j-\psi_j)^-]^4 + \frac{1}{3} [(y_j-\psi_j)^-]^4 - \frac{4}{3} (y_j - \psi_j)   [(\tilde{y}_j-\psi_j)^-]^3 \} \nonumber \\
& \geq & \eta (\tilde{y}_j - y_j) \{ [(\tilde{y}_j-\psi_j)^-]^3 - [(y_j-\psi_j)^-]^3 \}, \label{case2}
\end{eqnarray}
since $\eta \leq 0$ and  $\frac{4}{3} |a|^3 |b|  \leq  a^4 + \frac{1}{3} b^4$ in view of Young's inequality. \\
(iii) $j \in \mathcal N^0$: In this case we have $p_j \geq 0$ by (\ref{foc4}) and therefore
\begin{equation} \label{case3}
p_j r_j \geq 0 \geq \eta (\tilde{y}_j - y_j) \{ [(\tilde{y}_j-\psi_j)^-]^3 - [(y_j-\psi_j)^-]^3 \}.
\end{equation}

\noindent
Combining (\ref{case1})--(\ref{case3}) with (\ref{equation: 3}), (\ref{equation: 4}) and the definition of $\lambda_1$ we derive
\begin{eqnarray*}
\gamma^3 \sum_{j=1}^n p_j r_j m_j  & \geq  &  \eta \gamma^3 \sum_{j=1}^n (\tilde{y}_j - y_j) \{ [(\tilde{y}_j-\psi_j)^-]^3 - [(y_j-\psi_j)^-]^3 \} m_j \\
 & = &  \eta \gamma^3 \int_{\Omega} I_h \{ ([( \tilde y_h - \phi)^-]^3 - [(y_h - \psi)^-]^3) (\tilde y_h - y_h) \} dx \\
& = & -\eta  \int_{\Omega} | \nabla ( \tilde y_h - y_h) |^2 \, dx  + \eta \int_{\Omega} (v-u_h) (\tilde y_h - y_h)  \, dx   \\
& \geq & | \eta | \lambda_1 \int_{\Omega} | \tilde{y}_h - y_h |^2 \, dx - | \eta | \int_{\Omega} (v-u_h) (\tilde{y}_h - y_h ) \, dx.
\end{eqnarray*}
If we insert this bound into (\ref{jgdif2}) we deduce that
\begin{displaymath}
 J^{\gamma}_h(v) - J^{\gamma}_h(u_h) \geq \int_{\Omega} [ (\frac{1}{2} + | \eta | \lambda_1) | \tilde{y}_h - y_h |^2 + \frac{\alpha}{2} | v-u_h|^2 
 - | \eta | (v-u_h) (\tilde{y}_h - y_h) ] \, dx.
\end{displaymath}
It is not difficult to verify that the bilinear form $(x_1,x_2) \mapsto (\frac{1}{2} + | \eta | \lambda_1) x_1^2 + \frac{\alpha}{2} x_2^2 
 - | \eta | x_1 x_2$ is positive semidefinite (positive definite) if 
\begin{displaymath}
\frac{\alpha}{4} + \frac{\alpha}{2} \lambda_1 | \eta | - \frac{| \eta |^2}{4} \geq 0 \quad  (>0).
\end{displaymath}
This is the case, if $| \eta | \leq \mu$ ($ | \eta | < \mu$), where $\mu=\alpha \lambda_1 + \sqrt{ \alpha^2 \lambda_1^2 + \alpha}$ is the positive
 root of the quadratic equation $x^2 -2 \alpha \lambda_1 x - \alpha=0$. This completes the proof of the theorem. 
 \end{proof}
 
\section{Convergence}

Let $(\mathcal T_h)_{0 < h \leq h_0}$ be a sequence of triangulations of $\bar{\Omega}$ with mesh size $h:=\max_{T \in \mathcal T_h} h_T$, where $h_T= \mbox{diam}(T)$. We suppose that the
sequence is regular in the sense that there exists $\rho>0$ such that
\begin{displaymath}
r_T \geq \rho h_T \quad \forall T \in \mathcal T_h, \; 0< h \leq h_0,
\end{displaymath}
where $r_T$ denotes the radius of the largest ball contained in $T$. For a sequence  $(\gamma_h)_{0 < h \leq h_0}$  satisfying
\begin{equation} \label{gammah}
\gamma_h \rightarrow \infty \quad \mbox{ as } \quad  h \rightarrow 0
\end{equation}
we now consider the corresponding sequence of control problems $(\mathbb{P}^{\gamma_h}_h)$. Our first result is concerned with the question how the stationarity conditions at the
discrete level transfer to the continuous level under the assumption that the quantity defined in (\ref{defeta}) is uniformly bounded.

\begin{theorem}  \label{stationarylimit}
Let $(\bar u_h,\bar y_h,\bar p_h)_{0 < h \leq h_0} \subset L^2(\Omega) \times X_{h0} \times X_{h0}$ be a sequence of solutions of (\ref{foc1})--(\ref{foc3}) with 
corresponding $\eta_h \leq 0$ given by (\ref{defeta}) and suppose that 
\begin{equation} \label{upbound}
\Vert \bar u_h \Vert_{L^2(\Omega)} \leq C, \quad | \eta_h | \leq C, \qquad 0< h \leq h_0
\end{equation}
for some $C \geq 0$. Then there exists a subsequence $h \rightarrow 0$ and $(\bar u,\bar y,\bar p) \in L^2(\Omega) \times H^1_0(\Omega) \times H^1_0(\Omega)$ as well
as $\bar \xi, \bar \mu \in H^{-1}(\Omega)$ such that
\begin{displaymath}
\bar u_h \rightarrow \bar u \mbox{ in } L^2(\Omega), \quad \bar y_h \rightarrow \bar y \mbox{ in } H^1(\Omega), \quad \bar p_h \rightharpoonup  \bar p \mbox{ in } H^1(\Omega), \quad \eta_h \rightarrow \eta
\end{displaymath}
and
\begin{eqnarray}
\int_{\Omega} \nabla \bar y \cdot \nabla \phi \, dx =  \int_{\Omega} (f+ \bar u) \phi \, dx  + \langle \bar \xi, \phi \rangle\quad \forall \phi \in H^1_0(\Omega)  &&  \label{com1c} \\
\bar y   \geq \psi \mbox{ a.e. in } \Omega, \; \bar  \xi \geq 0, \;  \langle \bar \xi, \bar y-\psi \rangle  =  0 &&  \label{com2c} \\
\int_{\Omega} \nabla \bar p \cdot \nabla \phi \, dx  = \int_{\Omega} (\bar y-y_0) \phi \, dx   - \langle \bar \mu, \phi \rangle \quad \forall \phi \in H^1_0(\Omega) &&  \label{com3c} \\
\langle \bar \xi, \bar p \rangle =0, \, \langle \bar \mu, \bar y-\psi \rangle =0 &&  \label{com4c} \\
\alpha \bar u+ \bar  p =0 \quad \mbox{ a.e. in } \Omega \label{com5c} \\
\bar p \geq \eta (\bar y-\psi) \mbox{ a.e. in } \Omega, \; \bar \mu \geq \eta \, \bar \xi. \label{com6c} 
\end{eqnarray}
\end{theorem}
\begin{proof} Let us first derive on upper bound on $\bar y_h$ in $H^1(\Omega)$. Inserting $\phi_h= \bar y_h - I_h \psi$ into (\ref{equation: 3}) we derive
\begin{displaymath}
\int_{\Omega} \nabla \bar y_h \cdot \nabla (\bar y_h - I_h \psi) \, dx + \gamma_h^3 \int_{\Omega} I_h \{ [(\bar y_h- \psi)^-]^4 \} \, dx = \int_{\Omega} (f+\bar u_h) (\bar y_h - I_h \psi) \, dx,
\end{displaymath}
from which we infer with the help of Poincar\'e's inequality
\begin{equation}  \label{estyh}
\Vert \bar y_h \Vert^2_{H^1(\Omega)} + \gamma_h^3 \int_{\Omega} I_h \{ [(\bar y_h- \psi)^-]^4 \} \, dx \leq c \bigl( \Vert f \Vert^2_{L^2(\Omega)} + \Vert I_h \psi \Vert^2_{H^1(\Omega)} + \Vert \bar u_h \Vert^2_{L^2(\Omega)} \bigr) \leq c
\end{equation}
in view of (\ref{upbound}) and the fact that $\psi \in H^2(\Omega)$.
Next, using $\phi_h=p_h$ in (\ref{foc2}) we derive
\begin{displaymath}
\int_{\Omega} | \nabla \bar p_h |^2 \, dx + 3 \gamma_h^3 \int_{\Omega} I_h \{ [(\bar y_h - \psi)^-]^2 \bar p_h^2 \} \, dx = \int_{\Omega} (\bar y_h -y_0) \bar p_h \, dx,
\end{displaymath}
which combined with (\ref{estyh}) yields
\begin{equation} \label{phapriori1}
\Vert \bar p_h \Vert^2_{H^1(\Omega)} + \gamma_h^3 \int_{\Omega} I_h \{ [(\bar y_h - \psi)^-]^2 \bar p_h^2 \} \, dx  \leq c.
\end{equation}
Thus, there exists a subsequence $h \rightarrow 0$ and $\bar u \in L^2(\Omega)$, $\bar y \in H^1_0(\Omega)$,
$\bar p \in H^1_0(\Omega)$ as well as $\eta \leq 0$ such that
\begin{eqnarray}
\bar y_h & \rightharpoonup & \bar y \quad \mbox{ in } H^1_0(\Omega), \quad  \bar y_h \rightarrow \bar y \quad \mbox{ in } L^2(\Omega),  \label{yconv1} \\
\bar p_h & \rightharpoonup & \bar p \quad \mbox{ in } H^1_0(\Omega), \quad \bar p_h \rightarrow \bar p \quad \mbox{ in } L^2(\Omega), \label{pconv1} \\
\bar u_h & \rightarrow & \bar u \quad \mbox{ in } L^2(\Omega), \label{uconv1} \\
\eta_h & \rightarrow & \eta.  \label{etaconv1}
\end{eqnarray}
Note that the strong convergence in (\ref{uconv1}) is a consequence of (\ref{pconv1}) and (\ref{foc3}). 
Let us first verify that $\bar y \in K$. Using the convexity of $s \mapsto (s^-)^4$ we derive that
\begin{displaymath}
[(\bar y_h - I_h \psi)^-]^4 \leq I_h \{ [ (\bar y_h - \psi)^-]^4 \} \quad \mbox{ on } T\, \mbox{ for all } T \in \mathcal T_h,
\end{displaymath}
which together with (\ref{estyh}) yields
\begin{displaymath}
\Vert (\bar y_h - I_h \psi)^- \Vert_{L^1(\Omega)} \leq c \Vert (\bar y_h - I_h \psi)^- \Vert_{L^4(\Omega)} \leq c \gamma_h^{-\frac{3}{4}} \rightarrow 0, \; h \rightarrow 0.
\end{displaymath}
We then have
\begin{displaymath}
\Vert  (\bar y - \psi)^- \Vert_{L^1(\Omega)}  \leq \Vert (\bar y_h - I_h \psi)^- \Vert_{L^1(\Omega)} + \Vert \bar y_h - \bar y \Vert_{L^1(\Omega)} + \Vert I_h \psi - \psi \Vert_{L^1(\Omega)}
\end{displaymath}
so that we deduce that $\Vert  (\bar y - \psi)^- \Vert_{L^1(\Omega)}=0$ by sending $h \rightarrow 0$. Thus $\bar y \geq \psi$ a.e. in $\Omega$ and hence $\bar y \in K$. \\
\noindent
We next show that $\bar y$ is the solution of (\ref{obstacle}) with $u=\bar u$ by verifying that $\bar y$ minimizes the functional
\begin{displaymath}
y \mapsto \frac{1}{2} \int_{\Omega} | \nabla y |^2 \, dx - \int_{\Omega} (f+ \bar u) y \, dx \quad \mbox{ over } K.
\end{displaymath}
To see this, let $y \in K$ be arbitrary. Arguing as in the proof of Proposition 5.2 in \cite{B15} we obtain a sequence $(y_k)_{k \in \mathbb{N}}$ such that $y_k \in K \cap H^2(\Omega), k \in \mathbb{N}$ and
$y_k \rightarrow y$ in $H^1(\Omega)$ as $k \rightarrow \infty$. Since 
$\bar y_h$ is a solution of (\ref{equation: 3}) it satisfies $Q(\bar y_h) = \min_{z_h \in X_{h0}} Q(z_h)$, where
\begin{equation} \label{innermin}
Q(z_h)= \frac{1}{2} \int_{\Omega} | \nabla z_h |^2 \, dx + \frac{\gamma_h^3}{4} \int_{\Omega} I_h \{ [(z_h- \psi)^-]^4 \} \, dx - 
\int_{\Omega} (f+ \bar u_h) z_h \, dx.
\end{equation}
Therefore,
\begin{eqnarray}
\lefteqn{ \hspace{-2cm}  \frac{1}{2} \int_{\Omega} | \nabla \bar y_h |^2 \, dx + \frac{\gamma_h^3}{4} \int_{\Omega} I_h \{ [( \bar y_h- \psi)^-]^4 \} \, dx - 
\int_{\Omega} (f+ \bar u_h) \bar y_h \, dx } \nonumber  \\
& \leq & \frac{1}{2} \int_{\Omega} | \nabla I_h y_k |^2 \, dx - \int_{\Omega} (f+ \bar u_h) I_h y_k \, dx, \label{est1}
\end{eqnarray}
since $y_k(x_j) \geq \psi(x_j), j=1,\ldots,n$. 
Letting first $h \rightarrow 0$ and then $k \rightarrow \infty$ we infer that
\begin{displaymath} 
\frac{1}{2} \int_{\Omega} | \nabla \bar y |^2 \, dx - \int_{\Omega} (f+ \bar u) \bar y \, dx  \leq  \frac{1}{2} \int_{\Omega} | \nabla y |^2 \, dx - \int_{\Omega} (f+ \bar u)  y \, dx
\end{displaymath}
for all $y \in K$, so that $\bar y$ solves (\ref{obstacle}) with $u=\bar u$. If we use (\ref{est1}) for a 
sequence $\bar y_k \in K \cap H^2(\Omega)$ with $\bar y_k \rightarrow \bar y$ in $H^1(\Omega)$ we obtain
\begin{displaymath}
\frac{1}{2} \int_{\Omega} | \nabla \bar y_h |^2 \, dx + \frac{\gamma_h^3}{4} \int_{\Omega} I_h \{ [( \bar y_h- \psi)^-]^4 \} \, dx \leq 
\frac{1}{2} \int_{\Omega} | \nabla I_h \bar y_k |^2 \, dx + \int_{\Omega} (f+ \bar u_h)( \bar y_h - I_h \bar y_k) \, dx,
\end{displaymath}
from which we deduce that 
\begin{displaymath}
\limsup_{h \rightarrow 0} \bigl( \frac{1}{2} \int_{\Omega} | \nabla \bar y_h |^2 \, dx +  \frac{\gamma_h^3}{4} \int_{\Omega} I_h \{ [( \bar y_h- \psi)^-]^4 \} \, dx \bigr) 
 \leq  \frac{1}{2} \int_{\Omega} | \nabla \bar y_k |^2 \, dx + \int_{\Omega} (f+ \bar u)( \bar y - \bar y_k) \, dx
\end{displaymath}
and hence after sending $k \rightarrow \infty$
\begin{equation} \label{yconv3}
\limsup_{h \rightarrow 0} \bigl( \frac{1}{2} \int_{\Omega} | \nabla \bar y_h |^2 \, dx +  \frac{\gamma_h^3}{4} \int_{\Omega} I_h \{ [( \bar y_h- \psi)^-]^4 \} \, dx  \bigr)
 \leq   \frac{1}{2} \int_{\Omega} | \nabla \bar y |^2 \, dx.
\end{equation}
In particular, $\limsup_{h \rightarrow 0} \Vert \nabla \bar y_h \Vert_{L^2(\Omega)}^2 \leq \Vert \nabla \bar y \Vert_{L^2(\Omega)}^2$ and hence
$\Vert \nabla \bar y_h \Vert_{L^2(\Omega)}^2 \rightarrow \Vert \nabla \bar y \Vert_{L^2(\Omega)}^2$. Thus, we obtain together with (\ref{yconv3}) that
\begin{equation}  \label{yconv2}
\bar y_h \rightarrow \bar y \mbox{ in } H^1(\Omega) \quad \mbox{ and } \quad  \frac{\gamma_h^3}{4} \int_{\Omega} I_h \{ [( \bar y_h- \psi)^-]^4 \} \, dx \rightarrow 0.
\end{equation}
Next, let us introduce $\bar \xi, \bar \mu \in H^{-1}(\Omega)$ by 
\begin{eqnarray*}
\langle \bar \xi, \phi \rangle & = &  \int_{\Omega} \nabla \bar y \cdot \nabla \phi \, dx - \int_{\Omega} (f+ \bar u) \phi   \, dx, \\
\langle \bar \mu, \phi \rangle & = &  -\int_{\Omega} \nabla \bar p \cdot \nabla \phi \, dx + \int_{\Omega} (\bar y - y_0) \phi \, dx.
\end{eqnarray*}
Obviously, (\ref{com1c}) and (\ref{com3c}) are satisfied by definition, while (\ref{com2c}) follows from the fact that $\bar y$ is a solution of (\ref{obstacle}). Let us next show that (\ref{com4c}) holds. 
In view of (\ref{yconv2}), (\ref{pconv1}) and (\ref{foc2}) we have
\begin{eqnarray*}
| \langle \bar \mu, \bar y - \psi \rangle |  & \leftarrow & | \,  - \int_{\Omega} \nabla \bar p_h \cdot \nabla (\bar y_h - I_h \psi) \, dx + \int_{\Omega} (\bar y_h - y_0) (\bar y_h - I_h \psi) \, dx \, |  \\
& = & | \,  3 \gamma_h^3 \int_{\Omega} I_h \{ [(\bar y_h - \psi)^-]^3 \bar p_h \} \, dx \, | \\
& \leq & 3  \Bigl( \gamma_h^3 \int_{\Omega} I_h \{ [(\bar y_h - \psi)^-]^4 \} \, dx \Bigr)^{\frac{1}{2}}
\Bigl( \gamma_h^3 \int_{\Omega} I_h \{ [(\bar y_h - \psi)^-]^2 \bar p_h^2 \} \, dx \Bigr)^{\frac{1}{2}} \\
& \rightarrow & 0
\end{eqnarray*}
by  (\ref{yconv2}) and (\ref{phapriori1}). Hence $\langle \bar \mu, \bar y - \psi \rangle =0$ and in the same way we can show that $\langle \bar \xi, \bar p \rangle=0$. Furthermore,
(\ref{com5c}) is an immediate consequence of (\ref{foc3}).
It remains to prove (\ref{com6c}). Note that (\ref{foc4}) and the definition of $\eta_h$ imply
\begin{displaymath}
\bar p_h   \geq   \eta_h I_h [ (\bar y_h - \psi)^+] - \frac{ \eta_h }{3}  I_h [ (\bar y_h - \psi)^-] 
 =  \eta_h (\bar y_h - I_h \psi) - \frac{4}{3} \eta_h  I_h [ (\bar y_h - \psi)^-].
\end{displaymath}
Letting $h \rightarrow 0$ we find that $\bar p \geq \eta(\bar y - \psi)$ a.e. in $\Omega$, since $I_h[ (\bar y_h - \psi)^-] \rightarrow 0$ in $L^1(\Omega)$ in view of (\ref{yconv2}) 
and (\ref{upbound}).
Finally, let $\phi \in C^{\infty}_0(\Omega), \phi \geq 0$ be arbitrary. We then have by (\ref{foc2}) and (\ref{foc1}) that
\begin{eqnarray*}
\langle \bar \mu - \eta \bar \xi, \phi \rangle & \leftarrow & - \int_{\Omega} \nabla \bar p_h \cdot \nabla I_h \phi \, dx + \int_{\Omega}(\bar y_h - y_0) I_h \phi \, dx \\
& & - \eta_h  \int_{\Omega} \nabla \bar y_h \cdot \nabla I_h \phi \, dx + \eta_h \int_{\Omega} (f+ \bar u_h) I_h \phi \, dx \\
& = & 3 \gamma_h^3 \int_{\Omega} I_h \{ [(\bar y_h - \psi)^-]^2 \bar p_h I_h \phi \} \, dx + \eta_h \gamma_h^3 \int_{\Omega} I_h \{ [ (\bar y_h - \psi)^-]^3 I_h \phi \} \, dx \\
& = & \gamma_h^3 \, \sum_{\bar y_j < \psi_j} [(\bar y_j - \psi_j)^-]^2 \{ 3 \bar p_j + \eta_h (\bar y_j - \psi_j) \} \phi(x_j) m_j \\
& \geq & 0,
\end{eqnarray*}
by the definition of $\eta_h$ and since $\phi(x_j) \geq 0$. Hence $\bar \mu - \eta \bar \xi \geq 0$ and (\ref{com6c}) holds. 
\end{proof}

\noindent
In order to relate the system (\ref{com1c})--(\ref{com6c}) to known stationarity concepts we briefly recall the notion of strong stationarity:

\begin{definition} \label{defstrong}
The point $(u,y,\xi) \in L^2(\Omega) \times H^1_0(\Omega) \times H^{-1}(\Omega)$ is called strongly stationary if there exists $p \in H^1_0(\Omega)$ such that
\begin{eqnarray}
\int_{\Omega} \nabla  y \cdot \nabla \phi \, dx =  \int_{\Omega} (f+ u) \phi \, dx  + \langle \xi, \phi \rangle\quad \forall \phi \in H^1_0(\Omega)  &&  \label{strong1} \\
 y   \geq \psi \mbox{ a.e. in } \Omega, \;  \xi \geq 0, \;  \langle \xi,  y-\psi \rangle  =  0 &&  \label{strong2} \\
p \in S_y, \quad \int_{\Omega} \nabla  p \cdot \nabla \phi \, dx  \leq  \int_{\Omega} (y-y_0) \phi \, dx   \quad \forall \phi \in S_y &&  \label{strong3} \\
\alpha  u+ p =0 \quad \mbox{ a.e. in } \Omega, \label{strong4}
\end{eqnarray}
where $S_y= \lbrace \phi \in H^1_0(\Omega) \, | \, \phi \geq 0 \mbox{ q.e. on } Z_y, \langle \xi, \phi \rangle =0 \rbrace$ and
$Z_y=\lbrace x \in \Omega \, | \, y(x)=\psi(x) \rbrace$ (defined up to sets of zero capacity). Here,  q.e. stands for quasi-everywhere.
\end{definition}
\noindent 
It is shown in \cite[Theorem 2.2]{MP84} that a solution of $(\mathbb{P})$ is strongly stationary. This result was extended to the case of control constraints in
\cite{W14}. In our next result we show that a solution of the system (\ref{com1c})--(\ref{com6c}) is strongly stationary. Furthermore, we prove a continuous analogue of Theorem \ref{main2}.

\begin{theorem} \label{statprop}
Suppose that  $(u,y,p,\xi,\mu) \in L^2(\Omega) \times H^1_0(\Omega) \times H^1_0(\Omega) \times H^{-1}(\Omega) \times H^{-1}(\Omega)$ is a solution of (\ref{com1c})--(\ref{com6c}) for some $\eta \leq 0$. Then there holds: \\
a) $(u,y,\xi)$ is strongly stationary. \\[2mm]
b) If
\begin{equation} \label{optcond}
| \eta | \leq \alpha \lambda_1 + \sqrt{ \alpha^2 \lambda_1^2 + \alpha},
\end{equation}
then $ u$ is a global minimum for Problem~$(\mathbb{P})$. If the inequality (\ref{optcond}) is strict, then
$ u$ is the unique global minimum. 
\end{theorem}
\begin{proof} a) We only have to prove (\ref{strong3}). To do so, we shall make use of some basic results from capacity theory for which we refer  the reader to \cite[Section 2]{W14}. 
Since $p \geq \eta(y-\psi) \mbox{ a.e. in }\Omega$ we deduce from \cite[Lemma 2.3]{W14} that $p \geq \eta(y-\psi) \mbox{ q.e. on }\Omega$ and hence that $p \geq 0 \mbox{ q.e. on }Z_y$.
Combining this relation with (\ref{com4c}) we infer that $p \in S_y$. Next, (\ref{com6c}) implies that $\lambda:= \mu - \eta \xi$ is a nonnegative functional in $H^{-1}(\Omega)$. Hence there exists a
regular Borel measure (also denoted by $\lambda$) such that 
\begin{displaymath}
\langle \lambda, \phi \rangle = \int_{\Omega} \phi \, d \lambda, \quad \phi \in H^1_0(\Omega),
\end{displaymath}
where we integrate the quasi-continuous representative of $\phi$. Lemma 2.4 in \cite{W14} yields that $y \geq \psi \, \lambda$-a.e. in $\Omega$ so that
we have for every compact set $C \subset \Omega$
\begin{displaymath}
0 \leq \int_C (y - \psi) \, d \lambda \leq \int_{\Omega} (y - \psi) \, d \lambda = \langle \lambda, y-\psi \rangle =0
\end{displaymath}
in view of (\ref{com2c}) and (\ref{com4c}). Thus $\lambda(C \cap \lbrace y> \psi \rbrace)=0$ for all compact sets $C \subset \Omega$ and hence $\lambda(\lbrace y> \psi \rbrace)=0$. We deduce
for every $\phi \in S_y$
\begin{displaymath}
\langle \mu, \phi \rangle = \langle \mu - \eta \xi, \phi \rangle = \int_{\Omega} \phi \, d \lambda = \int_{\lbrace y > \psi \rbrace} \phi \, d \lambda +
\int_{\lbrace y = \psi \rbrace} \phi \, d \lambda = \int_{Z_y} \phi \, d \lambda \geq 0,
\end{displaymath}
since $\phi \geq 0$ q.e. (and hence also $\lambda$-a.e.) on $Z_y$. Combining this result with (\ref{com3c}) we infer
\begin{displaymath}
\int_{\Omega} \nabla p \cdot \nabla \phi \, dx = \int_{\Omega}(y-y_0) \phi \, dx  - \langle \mu, \phi \rangle \leq \int_{\Omega}(y-y_0) \phi \, dx \qquad \forall \phi \in S_y,
\end{displaymath}
and hence (\ref{strong2}) is satisfied. \\[2mm]
b) Let $v \in L^2(\Omega)$ be arbitrary and $\tilde y \in K$ the solution of 
\begin{displaymath}
\int_{\Omega} \nabla \tilde y \cdot \nabla ( \phi - \tilde y) dx \geq \int_{\Omega} (f+v) (\phi - \tilde y) dx \qquad \forall \phi \in K.
\end{displaymath}
Defining $\tilde \xi \in H^{-1}(\Omega)$ by
\begin{displaymath}
\langle \tilde \xi, \phi \rangle:= \int_{\Omega} \nabla \tilde y \cdot \nabla \phi \, dx - \int_{\Omega} (f+v) \phi \, dx,
\end{displaymath}
we have that $\tilde \xi \geq 0$ and $\langle \tilde \xi, \tilde y - \psi \rangle =0$. Similarly as in the proof of Theorem \ref{main2} we calculate
\begin{equation} \label{dif1}
J(v) - J(u)  =  \frac{1}{2} \Vert \tilde y - y \Vert_{L^2(\Omega)}^2  + \frac{\alpha}{2} \Vert v - u \Vert_{L^2(\Omega)}^2 
 + \int_{\Omega}  (y-y_0) (\tilde y - y) \, dx  + \alpha \int_{\Omega} u (v-u) \, dx. 
\end{equation}
We infer from (\ref{com3c}), (\ref{com1c}) and the definition of $\tilde \xi$ that
\begin{eqnarray*}
 \int_{\Omega}  (y-y_0) (\tilde y - y) \, dx & =  & \int_{\Omega} \nabla p \cdot \nabla (\tilde y -y) \, dx + \langle \mu, \tilde y - y \rangle  \\
& = & \int_{\Omega} p (v-u) \, dx  + \langle \tilde \xi - \xi,p \rangle + \langle \mu, \tilde y - y \rangle.
\end{eqnarray*}
Inserting this relation into (\ref{dif1}) and recalling (\ref{com2c}), (\ref{com4c}) and (\ref{com6c}) we derive
\begin{eqnarray}
J(v) - J(u) & = & \frac{1}{2} \Vert \tilde y - y \Vert_{L^2(\Omega)}^2 + \frac{\alpha}{2} \Vert v - u \Vert_{L^2(\Omega)}^2
+ \langle \tilde \xi,p \rangle  + \langle \mu,\tilde y - \psi \rangle \nonumber  \\
& \geq &  \frac{1}{2} \Vert \tilde y - y \Vert_{L^2(\Omega)}^2 + \frac{\alpha}{2} \Vert v - u \Vert_{L^2(\Omega)}^2
 + \eta \, \langle \tilde \xi, y - \psi \rangle + \eta \, \langle \xi, \tilde y - \psi \rangle \label{dif3} \\
& = & \frac{1}{2} \Vert \tilde y - y \Vert_{L^2(\Omega)}^2 + \frac{\alpha}{2} \Vert v - u \Vert_{L^2(\Omega)}^2 - \eta \langle \tilde \xi - \xi,  \tilde y - y \rangle, \nonumber
\end{eqnarray} 
since $\tilde y - \psi \geq 0$ and $\langle  \tilde \xi,  \tilde y - \psi \rangle=0$. Using once more (\ref{com1c}), the definition of $\tilde \xi$ and recalling (\ref{lambda1}) we may write
\begin{eqnarray*}
\langle \tilde \xi - \xi,  \tilde y - y \rangle & = & \int_{\Omega} | \nabla( \tilde{y} - y) |^2 \, dx - \int_{\Omega} (\tilde{y} - y) (v-u) \, dx \\
& \geq & \lambda_1 \int_{\Omega} | \tilde{y} - y |^2 \, dx - \int_{\Omega} (\tilde{y} - y) (v-u) \, dx.
\end{eqnarray*}
If we multiply this relation by $-\eta= | \eta |$ and insert it into  (\ref{dif3}) we obtain
\begin{displaymath}
 J(v) - J(u) \geq \int_{\Omega} [ (\frac{1}{2} + | \eta | \lambda_1) | \tilde{y} - y |^2 + \frac{\alpha}{2} | v-u|^2 
 - | \eta | (v-u) (\tilde{y} - y) ] \, dx
\end{displaymath}
and the result follows in the same way as in the proof of Theorem \ref{main2}.
\end{proof} 

\noindent
As an immediate consequence we have:

\begin{corollary} \label{cor1}
Let $(\bar u_h,\bar y_h,\bar p_h)_{0<h \leq h_0}$ be a sequence of solutions of (\ref{foc1})--(\ref{foc3}) with 
corresponding $\eta_h \leq 0$ given by (\ref{defeta}) and suppose that
\begin{equation} \label{etahb}
| \eta_h | \leq \alpha \lambda_1 + \sqrt{ \alpha^2 \lambda_1^2 + \alpha}, \quad 0<h \leq h_0.
\end{equation}
Then
\[
\bar u_h \to \bar u \mbox{ in } L^2(\Omega) \mbox{ for a subsequence } h \to 0,
\]
where $\bar u$ is a global minimum for Problem~$(\mathbb{P})$. If 
\begin{equation} \label{strict}
\displaystyle
| \eta_h | \leq \kappa \alpha \bigl( \lambda_1 + \sqrt{ \alpha^2 \lambda_1^2 + \alpha} \bigr), \quad 0< h \leq h_0,
\end{equation}
for some $0 < \kappa <1$, then $\bar u$ is the unique global solution of $(\mathbb{P})$ and the whole sequence $( \bar u_h)_{0<h \leq h_0}$ converges 
to $\bar u$. 
\end{corollary}
\begin{proof} Let us
denote by $\hat y_h$  the solution of (\ref{equation: 3}) with $u \equiv 0$. In the same way as at the beginning of the proof of Theorem \ref{stationarylimit} we can show that  $\Vert \hat y_h \Vert_{H^1(\Omega)} \leq c$.
It follows from (\ref{etahb}) and Theorem \ref{main2} that $\bar u_h$ is a solution of $(\mathbb{P}^{\gamma_h}_h)$, so that in particular $J^{\gamma_h}_h(\bar u_h) \leq J^{\gamma_h}_h(0), \, 0 < h \leq h_0$ and therefore
\begin{equation} \label{uhapriori}
\frac{\alpha}{2} \Vert \bar u_h \Vert^2_{L^2(\Omega)} \leq \frac{1}{2} \Vert \hat{y}_h - y_0 \Vert^2_{L^2(\Omega)} \leq c.
\end{equation}
Combining (\ref{uhapriori}) with (\ref{etahb}) we may infer from Theorem \ref{stationarylimit} that there exists a subsequence $h \rightarrow 0$ and a solution 
$(\bar u,\bar y, \bar p, \bar \xi, \bar \mu)$ of (\ref{com1c})--(\ref{com6c}) such that 
\begin{displaymath}
\bar u_h \rightarrow \bar u \mbox{ in } L^2(\Omega), \; \bar y_h \rightarrow \bar y \mbox{ in } H^1(\Omega), \, \eta_h \rightarrow \eta \leq 0.
\end{displaymath}
Since $| \eta | \leq \alpha \lambda_1 + \sqrt{ \alpha^2 \lambda_1^2 + \alpha}$ it follows from Theorem \ref{statprop} b) that $\bar u$ is a global minimum for Problem~$(\mathbb{P})$.
If (\ref{strict}) holds, then the above inequality is strict and the minimum is unique.
\end{proof}

\section{The unpenalised case}\label{Sec4}
In this section we briefly discuss how our theory can be adapted to the case when the state is approximated by a discrete version of the variational inequality (\ref{obstacle}),
namely we consider the following discrete control problem:
\[
(\mathbb{P}_h) \quad
 \min_{u \in L^2(\Omega)} J_h(u) = \frac{1}{2} \Vert y_h - y_0 \Vert_{L^2(\Omega)}^2 +  \frac{\alpha}{2} \Vert u \Vert_{L^2(\Omega)}^2
\]
subject to
\begin{equation} \label{discobstacle}
y_h \in K_h, \quad \int_{\Omega} \nabla y_h \cdot \nabla ( \phi_h - y_h) dx \geq \int_{\Omega} (f+u) (\phi_h -y_h) dx \qquad \forall \phi_h \in K_h,
\end{equation}
where
\begin{displaymath}
K_h = \lbrace \phi_h \in X_{h0} \, | \, \phi_h(x) \geq (I_h \psi)(x), x \in \Omega \rbrace.
\end{displaymath}
Existence of a solution of $(\mathbb{P}_h)$ is shown in \cite[Section 3]{MT13}. We shall formulate the necessary first order optimality 
conditions in matrix/vector form. To do so, let us define the mass matrix $\mathcal M$ and the stiffness matrix $\mathcal A$, i.e.
\begin{displaymath}
\mathcal M_{ij}:= \int_{\Omega} \phi_i \phi_j \, dx, \quad \mathcal A_{ij}:= \int_{\Omega} \nabla \phi_i \cdot \nabla \phi_j \, dx, \quad i,j=1,\ldots,n.
\end{displaymath}
Introducing a slack variable $\bx \in \mathbb{R}^n$, problem (\ref{discobstacle}) can be written as 
\begin{eqnarray*}
\mathcal A \by = \Bigl( \int_{\Omega} (f+u) \phi_i \, dx \Bigr)_{i=1}^n + \bx &&   \\
y_j  \geq \psi_j, \xi_j \geq 0, \xi_j(y_j- \psi_j)  =  0, j=1,\ldots,n. &&  
\end{eqnarray*}
Here, $y_h= \sum_{j=1}^n y_j \phi_j$ and $\by = (y_j)_{j=1}^n$.  The following result is
proved in \cite[Theorem 4.1]{MT13}. Note that the
system given below slightly differs from the one in \cite{MT13} in that we have replaced $\bmu$ by $-\bmu$.

\begin{theorem} Let $u_h \in L^2(\Omega)$ be a local optimal solution of $(\mathbb{P}_h)$ with associated state $y_h  \in X_{h0}$
and slack variable $\bx \in \mathbb{R}^n$. Then there exist an adjoint state $p_h \in X_{h0}$ and a multiplier $\bmu \in \mathbb{R}^n$
such that the following strong stationarity system is satisfied:
\begin{eqnarray}
\mathcal A \by = \Bigl( \int_{\Omega} (f+u_h) \phi_i \, dx \Bigr)_{i=1}^n + \bx &&  \label{com1} \\
y_j  \geq \psi_j, \xi_j \geq 0, \xi_j (y_j- \psi_j)  =  0, j=1,\ldots,n &&  \label{com2} \\
\mathcal A^T \bp = \mathcal M \by - \Bigl( \int_{\Omega} y_0 \phi_i \, dx \Bigr)_{i=1}^n   - \bmu \label{com3} \\
(y_j - \psi_j) \mu_j= 0, \, \xi_j p_j = 0, j=1,\ldots,n \label{com4} \\
\alpha u_h+ p_h =0 \quad \mbox{ in } \Omega \label{com5} \\
\mu_k \geq 0, \, p_k \geq 0 \mbox{ for all } k \in \lbrace 1,\ldots,n \rbrace \mbox{ with } y_k-\psi_k=\xi_k=0. \label{com6} 
\end{eqnarray}
Here, $y_h=\sum_{j=1}^n y_j \phi_j, p_h=\sum_{j=1}^n p_j \phi_j$.
\end{theorem}

\noindent
In practice, the system (\ref{com1})--(\ref{com5}) can be solved with the help of 
a  primal--dual active set strategy, see \cite[Section 6]{IK00}. The corresponding numerical experiments indicate that this method typically enters into a cycle in the presence of bi-active sets. This is one of the 
reasons to base our numerical treatment on the Moreau-Yosida relaxed version $(\mathbb{P}^\gamma_h)$ of the original optimal control problem $(\mathbb{P})$.

\noindent
The following result is the analogue of Theorem \ref{main2}. We remark that the quantity $\eta$ in (\ref{com6a})  has been used in \cite{BM00} (see (4.19)) in order to show the equivalence between certain optimality systems.

\begin{theorem}  \label{main1}
Suppose that $(u_h,y_h,p_h,\bx,\bmu) \in L^2(\Omega) \times X_{h0}  \times X_{h0} \times \mathbb{R}^n \times \mathbb{R}^n$ is a solution of (\ref{com1})-(\ref{com6})
and define
\begin{equation} \label{com6a}
\eta := \min \Bigl( \min_{y_k > \psi_k} \frac{p_k}{y_k - \psi_k}, \min_{\xi_k>0} \frac{\mu_k}{\xi_k},0 \Bigr).
\end{equation}
If
\begin{equation} \label{optcondh}
| \eta | \leq \alpha \lambda_1 + \sqrt{ \alpha^2 \lambda_1^2 + \alpha},
\end{equation}
then $ u_h$ is a global minimum for Problem~$(\mathbb{P}_h)$. If the inequality (\ref{optcondh}) is strict, then
$ u_h$ is the unique global minimum. 
\end{theorem}
\begin{proof} The proof is essentially a discrete version of the proof of Theorem \ref{statprop} b). Note that (\ref{com6}) and the definition of $\eta$ imply that
\begin{displaymath}
p_j \geq \eta ( y_j - \psi_j), \, \, \mu_j \geq \eta \, \xi_j, \; j=1,\ldots,n
\end{displaymath}
so that $(u_h,y_h,p_h,\bx,\bmu)$ satisfies a discrete analogue of (\ref{com1c})--(\ref{com6c}). 
\end{proof}

\noindent
It is not difficult to see that the result of Theorem \ref{stationarylimit} also holds for a sequence of solutions $(\bar u_h,\bar y_h, \bar p_h)_{0<h \leq h_0}$ of 
(\ref{com1})-(\ref{com6}) satisfying the bounds (\ref{upbound}). The corresponding arguments in fact become a little bit easier because the penalisation term
is no longer present. Furthermore, the convergence result in Corollary \ref{cor1} holds as well. We omit the details.

\section{Numerical examples} \label{secn}
In this section we apply Theorem~\ref{main2} to some numerical examples taken from \cite{IK00}, \cite{MT13} and \cite{MRW15}. In Examples~1--3 the computational domain is given by $\Omega:=(0,1)\times(0,1)$, and consequently, the constant $\lambda_1$ from (\ref{lambda1}) has the value $\lambda_1=2\pi^2$. On the other hand, Example~4 is formulated on the L-shaped domain $\Omega:=(-1,0)\times (-1,1) \cup [0,1)\times (0,1)$ and $\lambda_1$ is approximated by solving a generalized eigenvalue problem leading to $\lambda_1 \approx 9.63977851$. In all of the examples, the domain $\Omega$ is partitioned using a uniform triangulation with mesh size $h=2^{-6}\sqrt{2}$.

\noindent
We solve (\ref{foc1}),(\ref{foc2}),(\ref{foc3}) using Newton's method with the stopping criterion
\[
\frac{1}{\alpha}\|p^{(k)}_h-p^{(k+1)}_h\|_{L^2(\Omega)} \leq 10^{-15},
\]
where $p^{(k)}_h$ denotes the discrete adjoint variable corresponding to the $k$-th iteration. We take the zero point as an initial guess for Newton's method and initialize our $\gamma-$homotopy with $\gamma=1$. As the value of $\gamma$ increases we take the solution of the system (\ref{foc1}),(\ref{foc2}),(\ref{foc3}) at the preceding value of $\gamma$ as starting value in the current Newton iteration.

\noindent
We introduce the sets of nodes 
\begin{align*}
&\mathcal{N}^+:=\{k \in \{1,\ldots,n\}: y_k-\psi_k > 0\},\\
&\mathcal{N}^0:=\{k \in \{1,\ldots,n\}: y_k-\psi_k =0\}, \text{ and}\\
&\mathcal{N}^-:=\{k \in \{1,\ldots,n\}: y_k-\psi_k < 0\}.
\end{align*}
Then, one can see from the quantity  $\min_{k \in \mathcal{N}^-} (y_k-\psi_k)$ the amount by which the state violates the obstacle constraint, which typically should tend to zero as the parameter $\gamma$ increases. We point out that the equality $y_k=\psi_k$ is often difficult to be observed on computers. In fact, it has never been detected when performing computations for the considered examples. Hence, we consider  directly  condition (\ref{optcondgammah}) as the set $\mathcal{N}^0$ is empty.  We shall   report for each example the values of
the quantities   $\eta$, $\min_{k \in \mathcal{N}^-} (y_k-\psi_k)$ and the number of Newton iterations,  denoted by $\# N$, as we increase the value of the penalization parameter $\gamma$. We stop increasing the parameter $\gamma$ once the linear system in Newton's method becomes too ill-conditioned. All the computations  are done using MATLAB~R2018a. 

In what follows we refer to the Lagrange interpolations of
\[
-\gamma^3 [(y_h-\psi)^-]^3 \text{ and } -3\gamma^3 [(y_h-\psi)^-]^2p_h
\]
as multipliers $\xi$ and $\mu$, respectively.

\begin{example} 
\label{example:IK4}
This is the Example~6.4 from \cite{IK00} with $f,y_0,\psi$ replaced by $-f$, $-y_0$ and $-\psi$. Here we choose the following data for $(\mathbb{P})$:
\begin{align*}
& \alpha=10^{-1}, \quad  y_0(x)=-(5x_1+x_2-1) \mbox{ in } \Omega, \quad  f(x)=-0.1 \mbox{ in } \Omega,\\
& \psi(x)=-4\big(x_1(x_1-1)+x_2(x_2-1)\big)-1.5  \mbox{ in } \Omega.
\end{align*}
We have 
\[
\alpha \lambda_1+\sqrt{\alpha^2 \lambda^2_1+\alpha} \approx 3.9730.
\]
The numerical results are reported in Table~\ref{table:IK4}.

We see that the condition  (\ref{optcondgammah}) holds for the considered values of $\gamma$, which indicates that the computed solution is the unique global minimum of $(\mathbb{P}^{\gamma}_h)$. Graphical illustration of the solution is provided in Figure~\ref{figure:ik4} for $\gamma=10^8$. We also observe that the violation of the obstacle constraint satisfies $\min_{k \in \mathcal{N}^-} (y_k-\psi_k) \sim -\gamma^{-1}$.

\begin{table}[h]
		\caption{Example~\ref{example:IK4} where $\alpha \lambda_1+\sqrt{\alpha^2 \lambda^2_1+\alpha} \approx 3.9730$.}
		\label{table:IK4}
		\begin{tabular}{ l  c  c c  c  c}
			\toprule
	    $\gamma$ &         $\min_{k \in \mathcal{N}^+} \frac{p_k}{y_k-\psi_k}$ &  $\min_{k \in \mathcal{N}^-} \frac{3p_k}{\psi_k-y_k}$&      $\eta$ &     $\min_{k \in \mathcal{N}^-} (y_k-\psi_k)$ & $\# N$ \\ 
		\midrule[1pt]
		1.0e+00&         -4.10942407e-05& 6.49648857e-01& -4.10942407e-05& -5.80096766e-01&    5\\
		1.0e+01&         -3.04816156e-04& 2.56513024e-01& -3.04816156e-04& -2.28879959e-01&    8\\
		1.0e+02&         -5.20715816e-04& 1.23095522e-01& -5.20715816e-04& -2.52527352e-02&   11\\
		1.0e+03&         -5.46497352e-04& 9.94266004e-02& -5.46497352e-04& -2.52516115e-03&   12\\
		1.0e+04&         -5.49770680e-04& 9.06835069e-02& -5.49770680e-04& -2.52508949e-04&   13\\
		1.0e+05&         -5.50113221e-04& 8.63210169e-02& -5.50113221e-04& -2.52508176e-05&   15\\
		1.0e+06&         -5.50622631e-04& 8.39380400e-02& -5.50622631e-04& -2.52524125e-06&   12\\
		1.0e+07&         -5.50679810e-04& 8.34709569e-02& -5.50679810e-04& -2.52538920e-07&   11\\
		1.0e+08&         -5.50685642e-04& 8.34241524e-02& -5.50685642e-04& -2.52542988e-08&   11\\
		1.0e+09&         -5.50686226e-04& 8.34194702e-02& -5.50686226e-04& -2.52543397e-09&   11\\
		1.0e+10&         -5.50686284e-04& 8.34190005e-02& -5.50686284e-04& -2.52543431e-10&   10\\
		1.0e+11&         -5.50686290e-04& 8.34189542e-02& -5.50686290e-04& -2.52543542e-11&   10\\
		1.0e+12&         -5.50686291e-04& 8.34185828e-02& -5.50686291e-04& -2.52542431e-12&   10\\
		1.0e+13&         -5.50686291e-04& 8.34167493e-02& -5.50686291e-04& -2.52520227e-13&    9\\
		1.0e+14&         -5.50686291e-04& 8.33984159e-02& -5.50686291e-04& -2.52575738e-14&    9\\
		1.0e+15&         -5.50686291e-04& 8.43368891e-02& -5.50686291e-04& -2.49800181e-15&    8\\

		\bottomrule 
	 \end{tabular}
\end{table}

	\begin{figure}
		
		\begin{subfigure}[b]{0.5\textwidth}
			\includegraphics[trim = 70mm 85mm 75mm 70mm, width=0.5\textwidth]{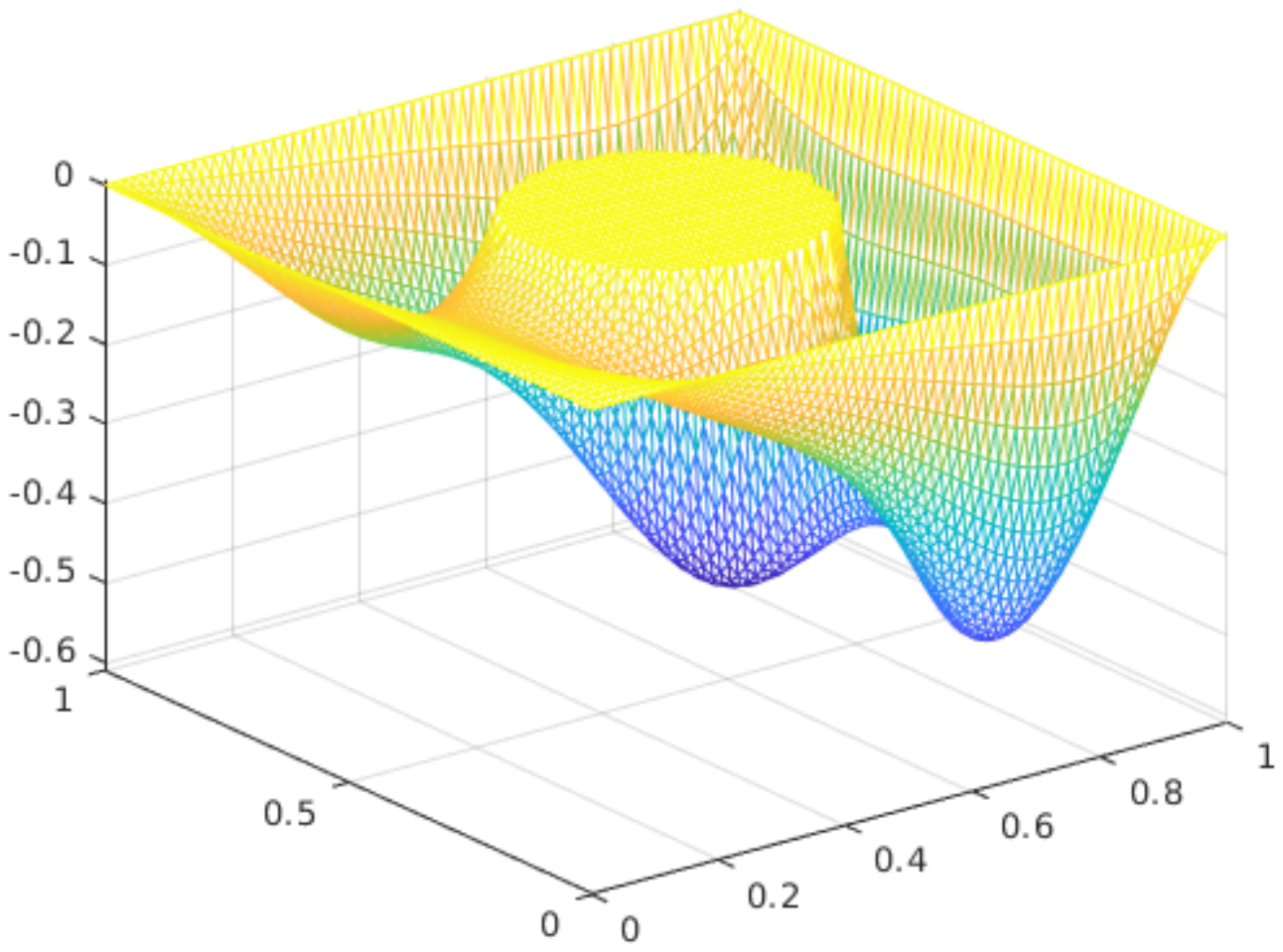}
			\caption{The optimal control}
		\end{subfigure}~ 
		\begin{subfigure}[b]{0.5\textwidth}
			\includegraphics[trim = 70mm 85mm 75mm 70mm, width=0.5\textwidth]{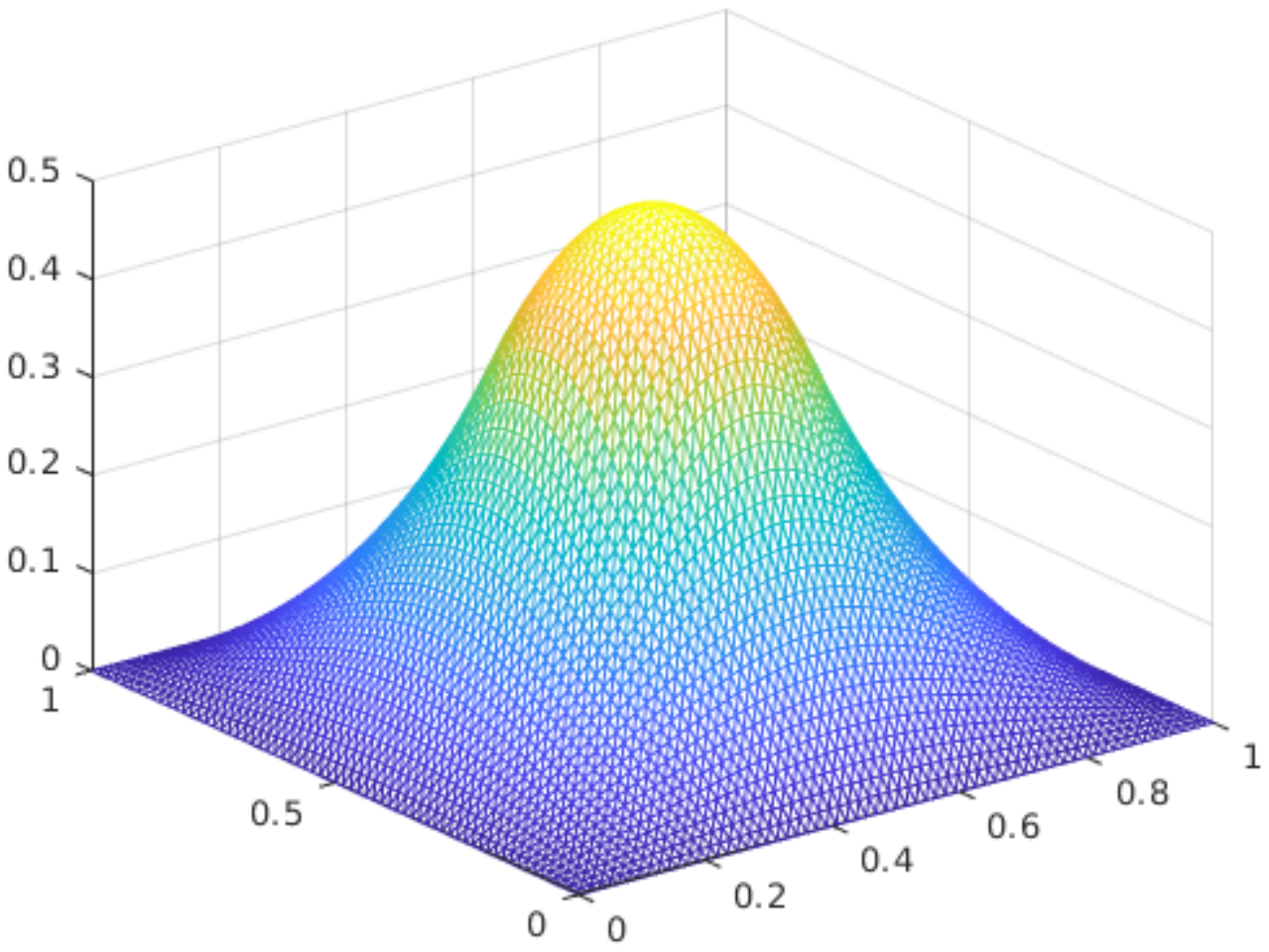}
			\caption{The optimal state}
		\end{subfigure}

		\begin{subfigure}[b]{0.5\textwidth}
			\includegraphics[trim = 70mm 85mm 75mm 70mm, width=0.5\textwidth]{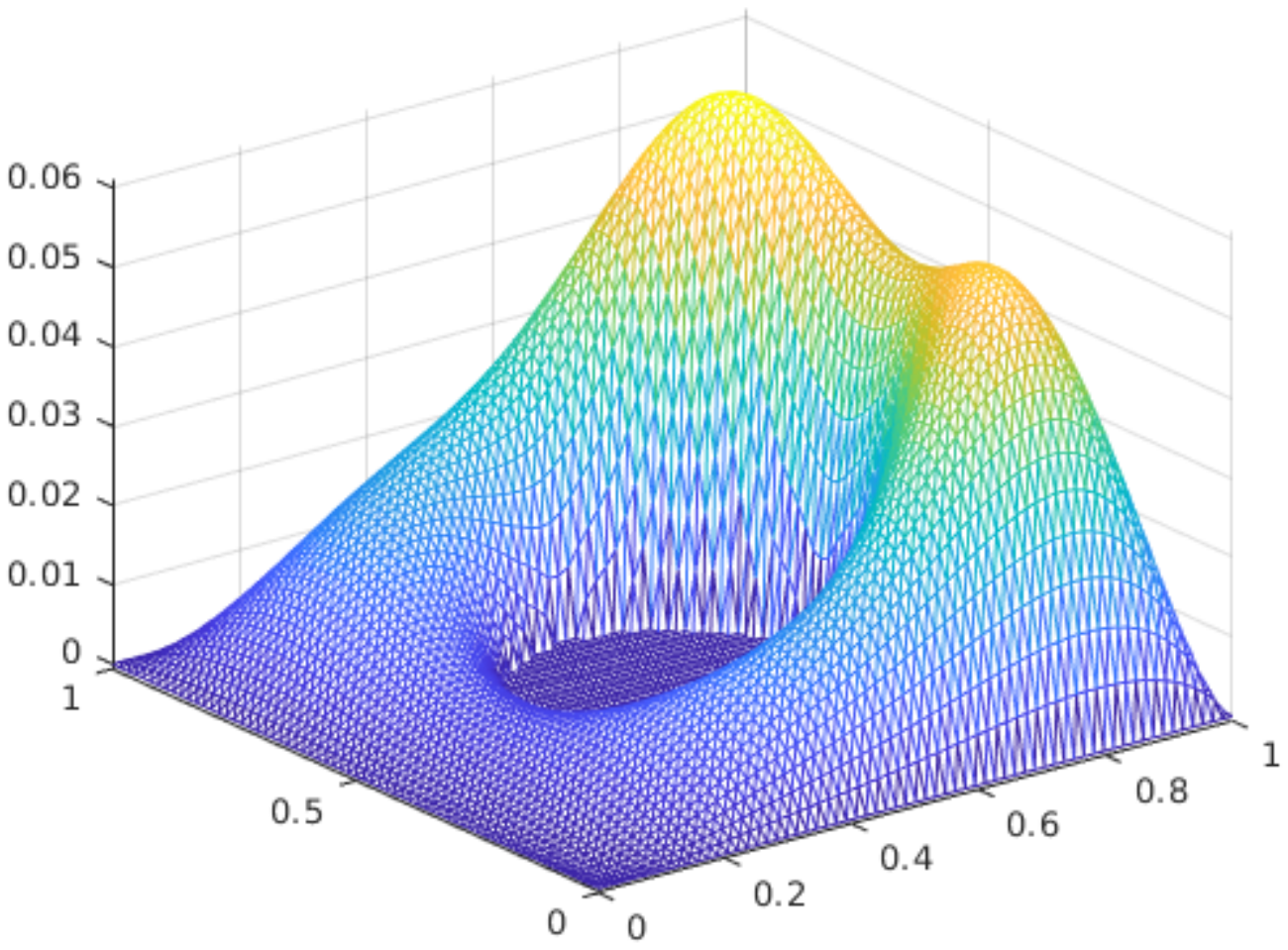}
			\caption{The optimal adjoint state}
		\end{subfigure}~
		\begin{subfigure}[b]{0.5\textwidth}
			\includegraphics[trim = 70mm 85mm 75mm 70mm, width=0.5\textwidth]{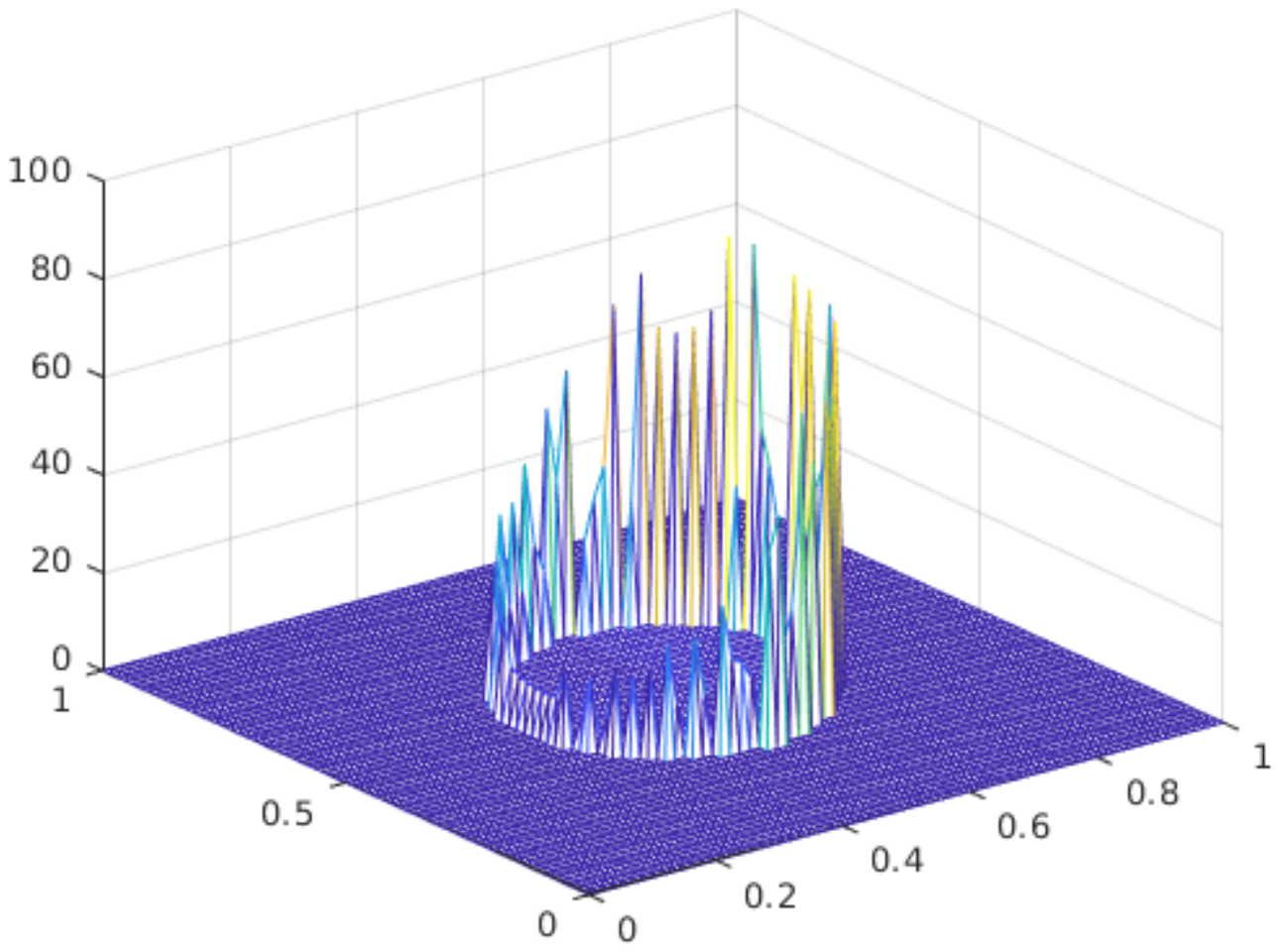}
			\caption{The multiplier $\mu$}
		\end{subfigure}
		\centering
		\begin{subfigure}[b]{0.5\textwidth}
			\includegraphics[trim = 70mm 85mm 75mm 70mm, width=0.5\textwidth]{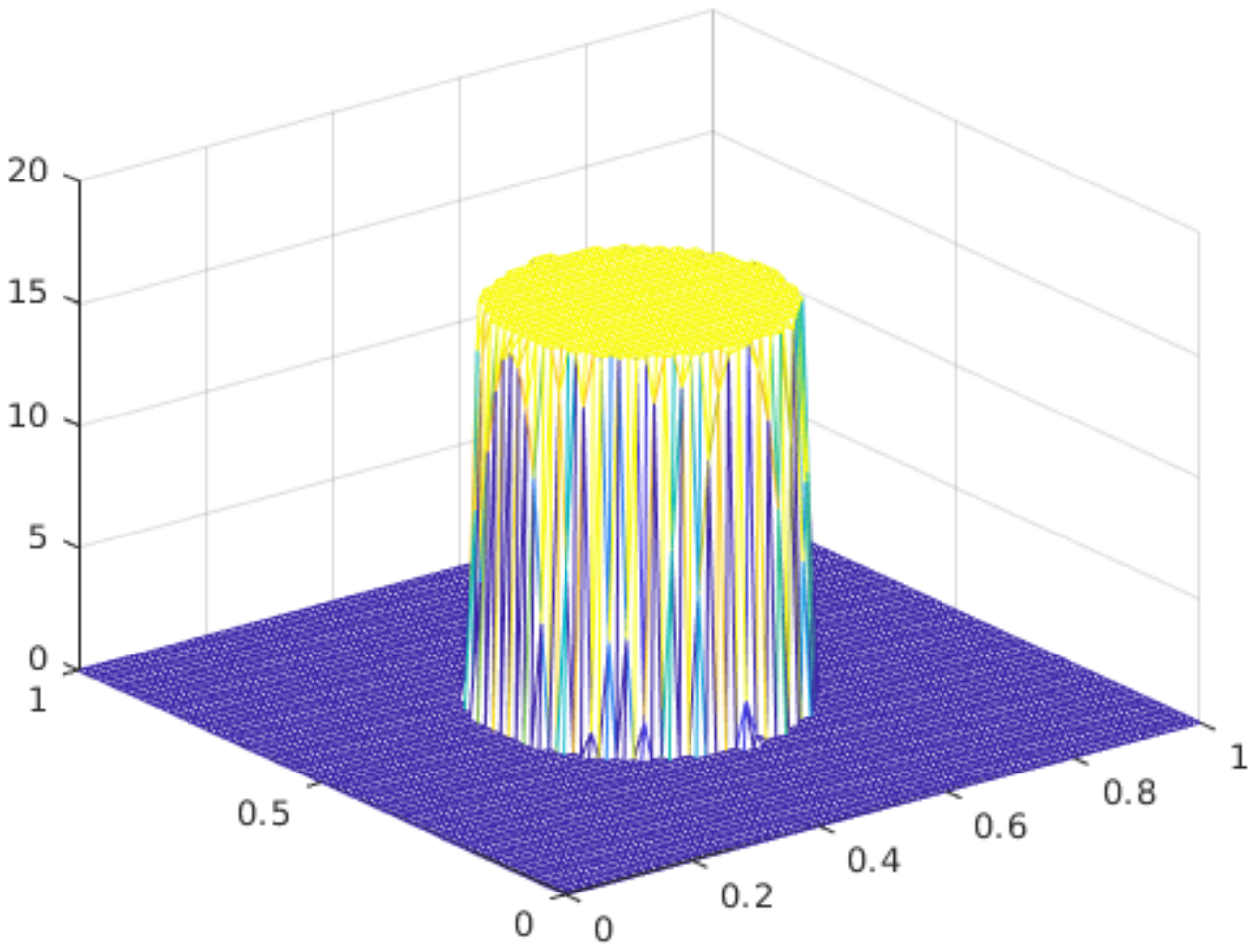}
			\caption{The multiplier $\xi$}
		\end{subfigure}
		
		\caption{Example~\ref{example:IK4}: The optimal control, state, adjoint state for $\gamma=10^8$.}\label{figure:ik4}
	\end{figure}
\end{example}

\begin{example}
\label{example:IK5}
This is  Example~6.5 from \cite{IK00} with lack of strict complementarity, where  $f,y_0$ are replaced by $-f$ and $-y_0$. Here we choose for $(\mathbb{P})$ the data  
\begin{align*}
& \alpha=10^{-1}, \quad  y_0(x)=-(5x_1+x_2-1) \mbox{ in } \Omega, \quad  f(x)=-(x_1-\frac{1}{2})\mbox{ in } \Omega, \quad \psi(x)=0 \mbox{ in } \Omega.
\end{align*}
We have 
\[
\alpha \lambda_1+\sqrt{\alpha^2 \lambda^2_1+\alpha} \approx 3.9730.
\]
The numerical results are reported in Table~\ref{table:ik5}. We see that the condition  (\ref{optcondgammah}) is satisfied for the considered values of $\gamma$, and hence, the unique global solution has been computed, which is presented in Figure~\ref{figure:ik5} for $\gamma=10^8$. We again observe that the violation of the obstacle constraint satisfies $\min_{k \in \mathcal{N}^-} (y_k-\psi_k) \sim -\gamma^{-1}$. This is also true in the subsequent examples.
	\begin{table}[h]
		\caption{Example~\ref{example:IK5} where $\alpha \lambda_1+\sqrt{\alpha^2 \lambda^2_1+\alpha} \approx 3.9730$.}
		\label{table:ik5}
		\begin{tabular}{ l  c  c  c  c  c}
			\toprule
			 $\gamma$ &         $\min_{k \in \mathcal{N}^+} \frac{p_k}{y_k-\psi_k}$ &  $\min_{k \in \mathcal{N}^-} \frac{3p_k}{\psi_k-y_k}$&      $\eta$ &     $\min_{k \in \mathcal{N}^-} (y_k-\psi_k)$ & $\# N$ \\  
			\midrule[1pt]
	      1.0e+00&         -9.18374766e-01& 5.31348016e+00& -9.18374766e-01& -8.34645660e-02&    4\\
	      1.0e+01&         -9.72376118e-01& 5.75321935e+00& -9.72376118e-01& -5.91117294e-02&    6\\
	      1.0e+02&         -1.01940110e+00 &8.79601682e+00& -1.01940110e+00& -7.28493127e-03&   11\\
	      1.0e+03&         -1.02120448e+00 &9.45008887e+00& -1.02120448e+00& -7.34501977e-04&   12\\
	      1.0e+04 &        -1.02054664e+00 &8.61664669e+00& -1.02054664e+00& -7.65785826e-05&   12\\
	      1.0e+05 &        -1.02044429e+00 &8.29383499e+00& -1.02044429e+00& -7.76577329e-06&   13\\
	      1.0e+06 &        -1.02030934e+00 &8.19782461e+00& -1.02030934e+00& -7.83603514e-07&   15\\
	      1.0e+07 &        -1.02028019e+00 &8.13627555e+00& -1.02028019e+00& -7.85170648e-08&   13\\
	      1.0e+08 &        -1.02027779e+00 &8.12975658e+00& -1.02027779e+00& -7.85327708e-09&   11\\
	      1.0e+09 &        -1.02027759e+00 &8.12910469e+00& -1.02027759e+00& -7.85343418e-10&   11\\
	      1.0e+10 &        -1.02027757e+00 &8.12903950e+00& -1.02027757e+00& -7.85344989e-11&   10\\
	      1.0e+11 &        -1.02027757e+00 &8.12903298e+00& -1.02027757e+00& -7.85345146e-12&   10\\
	      1.0e+12 &        -1.02027757e+00 &8.12903233e+00& -1.02027757e+00& -7.85345161e-13&   10\\
	       
			\bottomrule 
		\end{tabular}
	\end{table}
	
	\begin{figure}
		
		\begin{subfigure}[b]{0.5\textwidth}
			\includegraphics[trim = 70mm 85mm 75mm 70mm, width=0.5\textwidth]{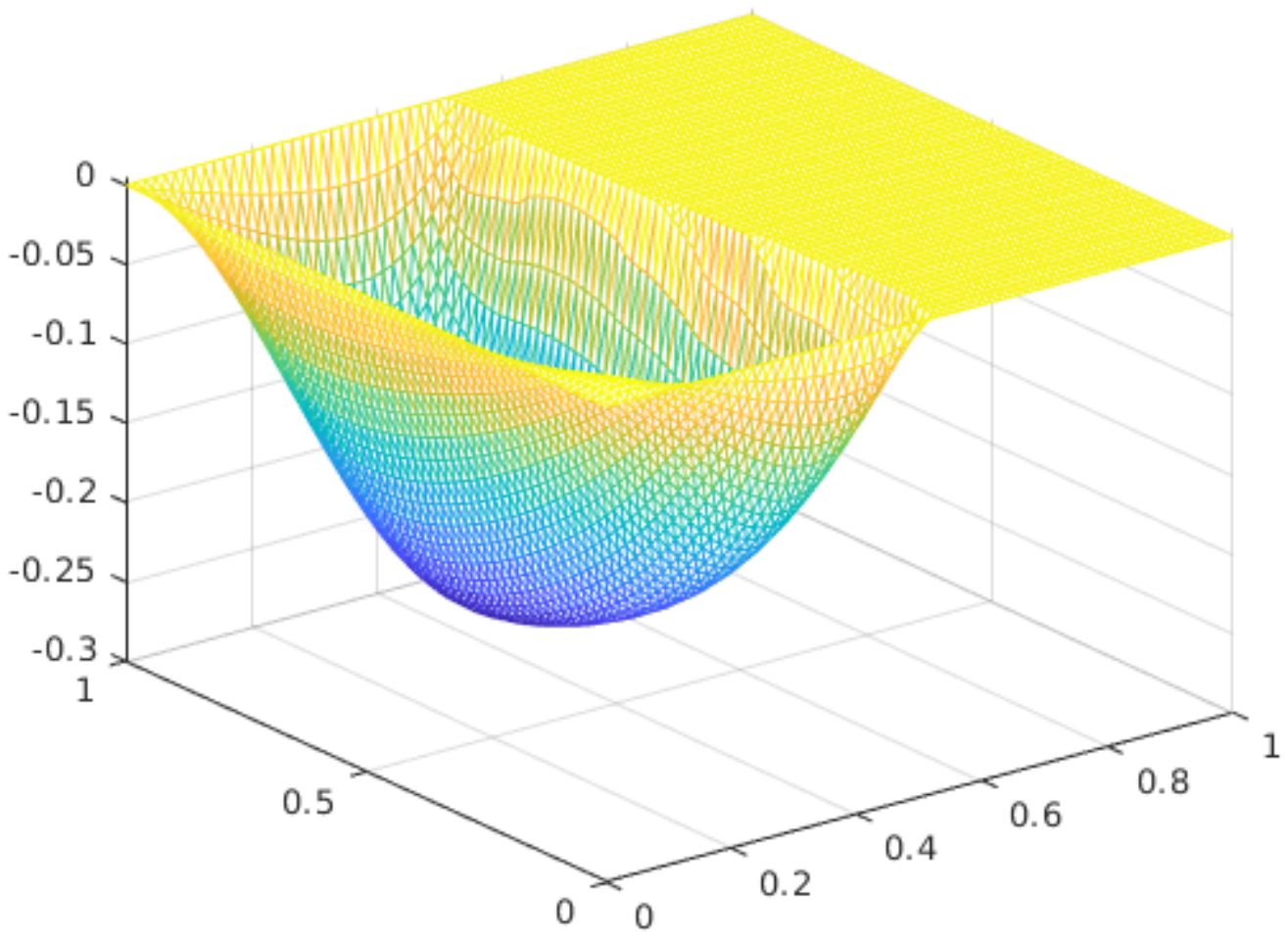}
			\caption{The optimal control}
		\end{subfigure}~ 
		\begin{subfigure}[b]{0.5\textwidth}
			\includegraphics[trim = 70mm 85mm 75mm 70mm, width=0.5\textwidth]{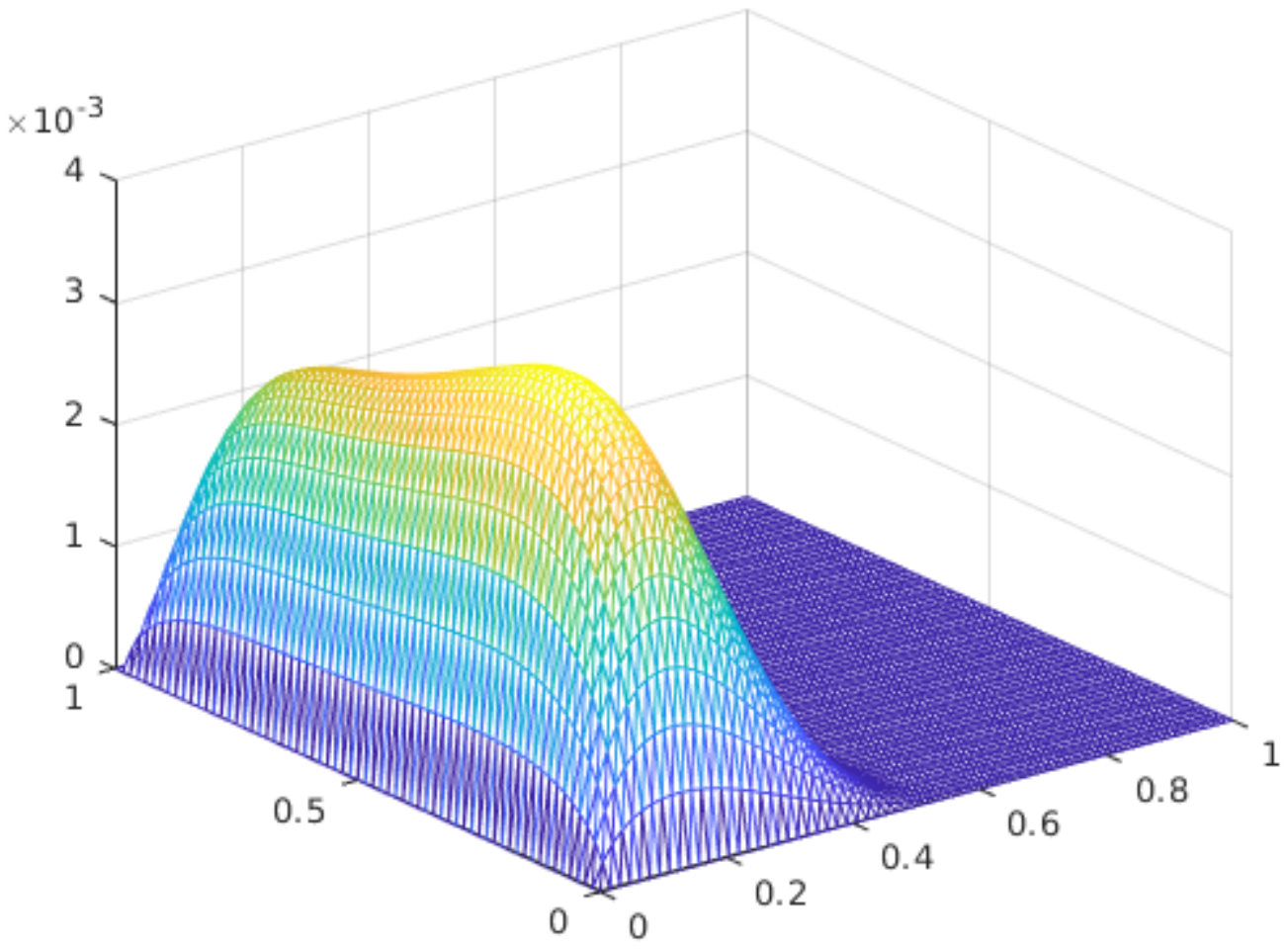}
			\caption{The optimal state}
		\end{subfigure}
		\begin{subfigure}[b]{0.5\textwidth}
			\includegraphics[trim = 70mm 85mm 75mm 70mm, width=0.5\textwidth]{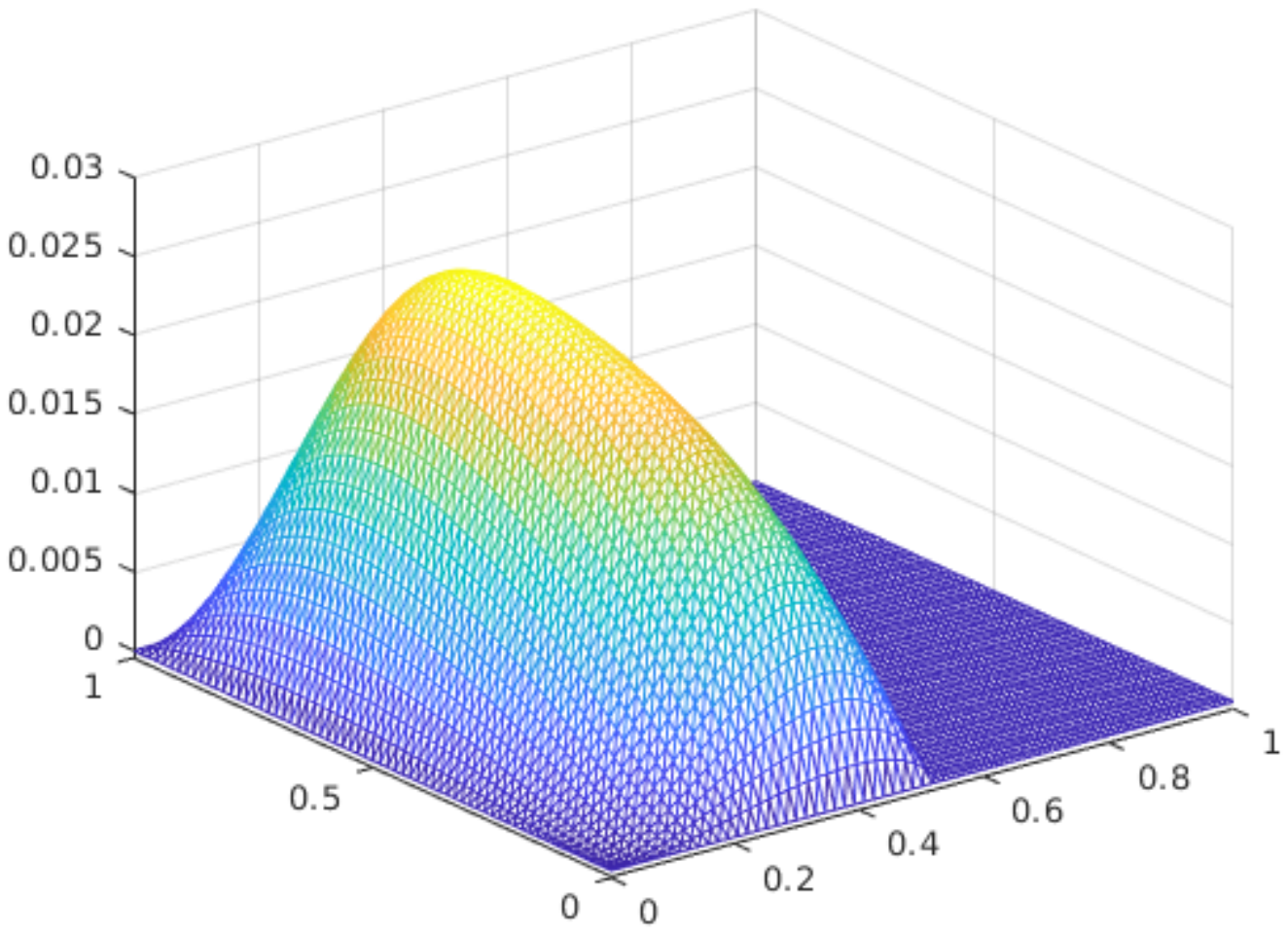}
			\caption{The optimal adjoint state}
		\end{subfigure}~
		\begin{subfigure}[b]{0.5\textwidth}
			\includegraphics[trim = 70mm 85mm 75mm 70mm, width=0.5\textwidth]{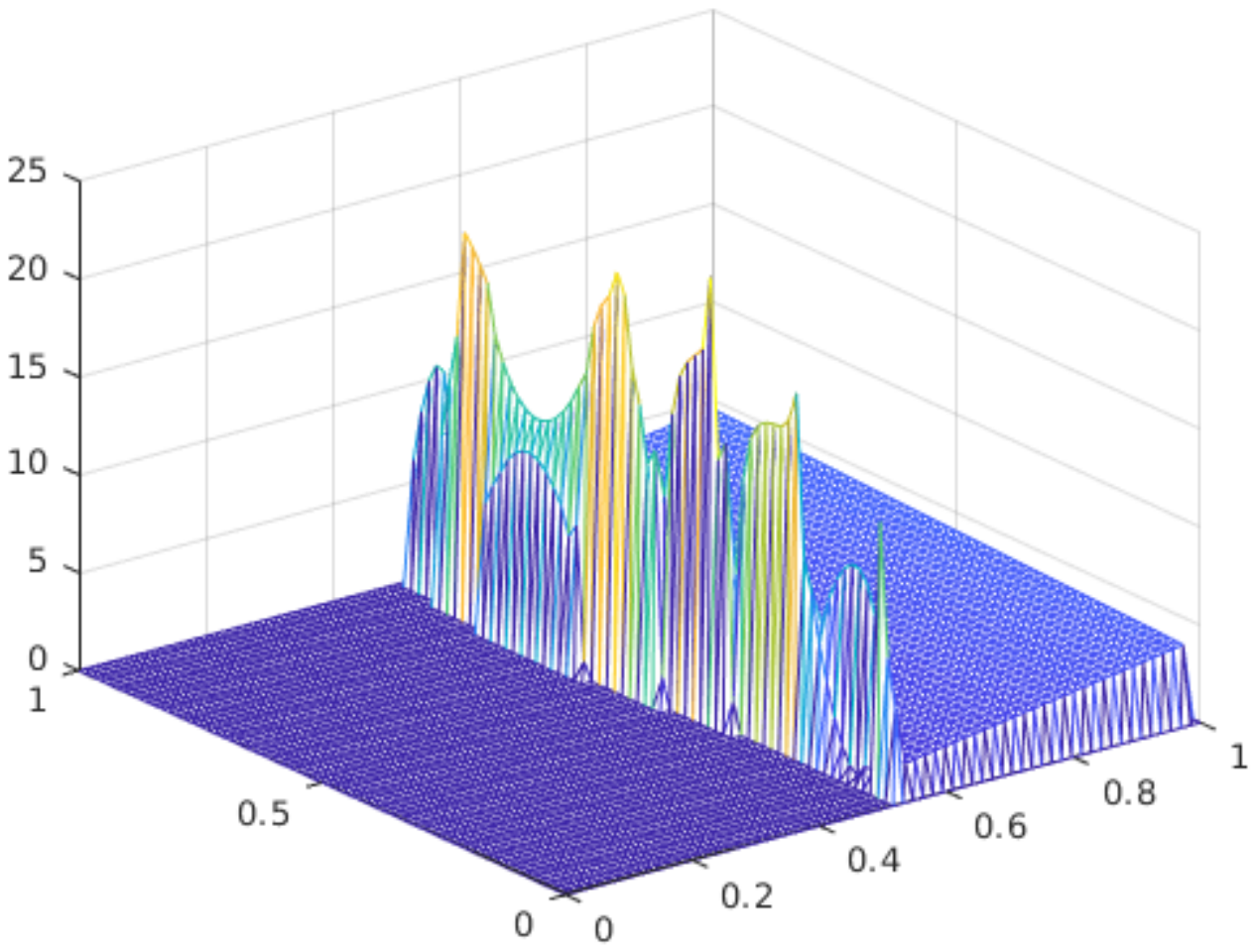}
			\caption{The multiplier $\mu$}
		\end{subfigure}
		\centering
		\begin{subfigure}[b]{0.5\textwidth}
			\includegraphics[trim = 70mm 85mm 75mm 70mm, width=0.5\textwidth]{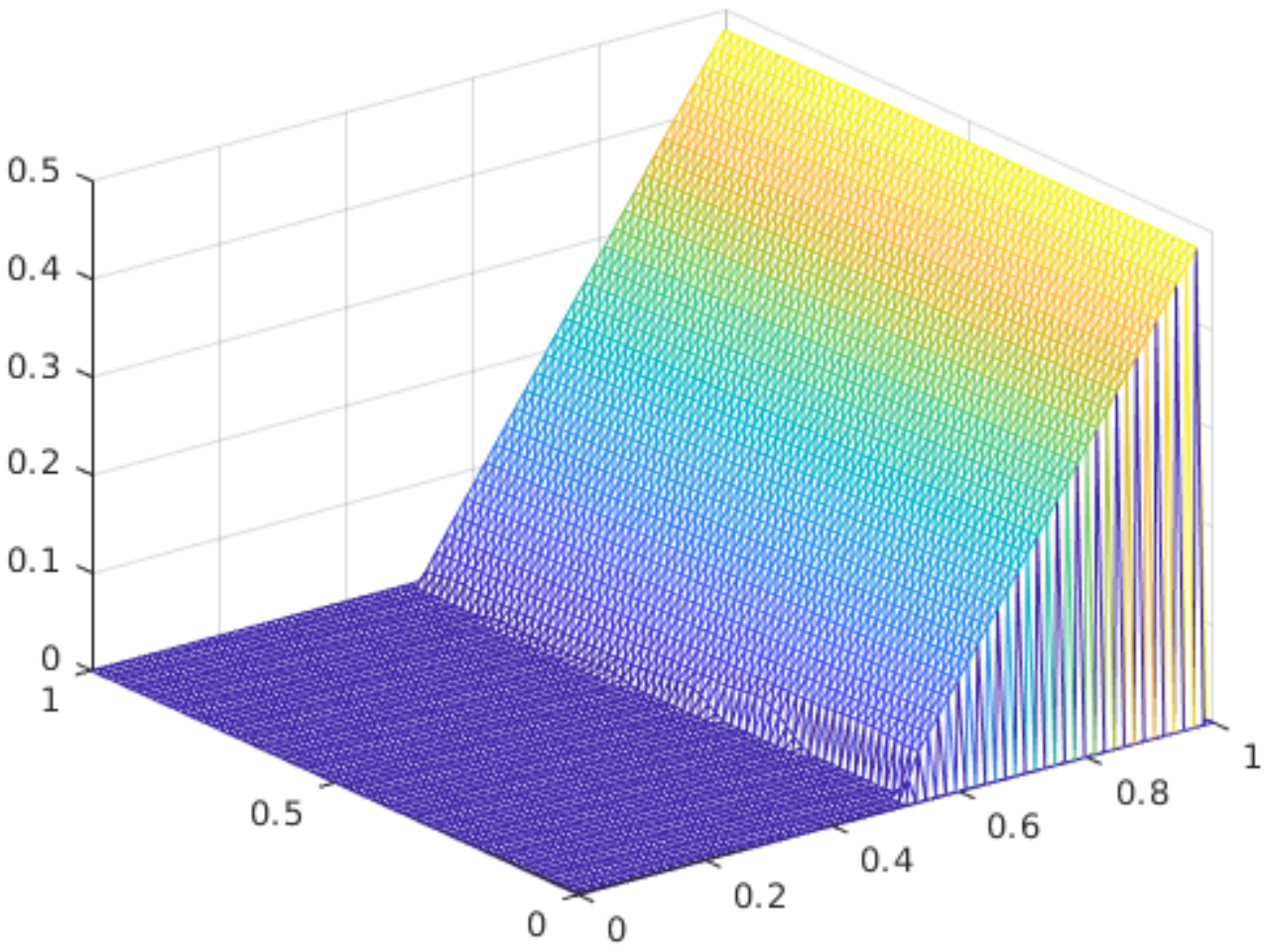}
			\caption{The multiplier $\xi$}
		\end{subfigure}

		\caption{Example~\ref{example:IK5}: The optimal control, state and adjoint state for $\gamma=10^8$.}\label{figure:ik5}
	\end{figure}

\end{example}

\begin{example} 
\label{example:mt1}
The data of this Example~1 is taken from \cite{MT13}, where  an exact solution for $(\mathbb{P})$ is constructed as follows;
\begin{align*}
	& \alpha=1,  \quad \psi=0 \mbox{ in } \Omega,\\
	& f(x)=-\Delta  y(x)-\xi(x) + \frac{1}{\alpha}  p(x)   \mbox{ in } \Omega\\
	& y_0(x)=\begin{cases}
	y(x)+ \Delta p_1(Q^tx)       & \quad \text{ in } \Omega_1,\\
	y(x)  & \quad \text{otherwise}.
	\end{cases}
\end{align*}
Here $y$ denotes the optimal state with the corresponding adjoint state  $p$ and slackness variable $\xi$, which are defined by 
\begin{align*}
	y(x_1,x_2)&=\begin{cases}
	y_1(x_1)\cdot y_2(x_2)       & \quad \text{ in } (0,0.5)\times (0,0.8),\\
	0  & \quad \text{otherwise},
	\end{cases}\\
	p(x)&=\begin{cases}
	p_1(Q^tx)       & \quad \text{ in } \Omega_1,\\
	0  & \quad \text{otherwise},
	\end{cases}\\
	\xi(x_1,x_2)&=\begin{cases}
	y_1(x_1-0.5) \cdot y_2(x_2)       & \quad \text{ in } (0.5,1)\times (0,0.8),\\
	0  & \quad \text{otherwise}.
	\end{cases}
\end{align*}
Moreover, $y_1$, $y_2$, and $p_1$ are the functions 
\begin{align*}
	& y_1(x_1)=-4096x_1^6+6144x_1^5-3072x_1^4+512x_1^3, \\
	& y_2(x_2)=-244.140625x_2^6+585.9375x_2^5-468.75x_2^4+125x_2^3,  \\
	& p_1(x_1,x_2)=(-200(x_1-0.8)^2+0.5)(-200(x_2-0.9)^2+0.5),
\end{align*}
and $\Omega_1$ is the  square with midpoint (0.8,0.9) and edge length 0.1 after being rotated by the matrix 
\[
Q=\begin{bmatrix}
\cos\frac{\pi}{6}& -\sin\frac{\pi}{6} \\
\sin\frac{\pi}{6}& \cos\frac{\pi}{6}
\end{bmatrix}
\]
around its midpoint. This example contains the biactive set
\[
\{x\in \Omega; y(x) = \xi(x) = 0\} \equiv [0, 1] \times [0.8, 1],
\]
which makes its numerical treatment challenging. Furthermore should we note that in this example the cost functional is of the form
\[
J(y,u)=\frac{1}{2}\|y-y_0\|_{L^2(\Omega)}^2+\frac{1}{2}\|u-u_d\|_{L^2(\Omega)}^2
\]
with $u_d=u+\frac{1}{\alpha}p$. We account for $u_d$ in the setting of $f$ above. Our theory also is valid in this situation without further modifications, see the proof of Lemma 2.1. 

From the previous data we have
\[
\alpha \lambda_1+\sqrt{\alpha^2 \lambda^2_1+\alpha} \approx 39.5037.
\]

We provide the numerical results in Table~\ref{table:mt1}. Again, the unique global minimum has been computed and the corresponding graphs  are illustrated in Figure~\ref{figure:mt1} when $\gamma=10^8$.
	
	\begin{table}[h]
		\caption{Example~\ref{example:mt1} with $\alpha \lambda_1+\sqrt{\alpha^2 \lambda^2_1+\alpha} \approx 39.5037$}.
		\label{table:mt1}
		\begin{tabular}{ l  c  c c c  c}
			\toprule
		 $\gamma$ &         $\min_{k \in \mathcal{N}^+} \frac{p_k}{y_k-\psi_k}$ &  $\min_{k \in \mathcal{N}^-} \frac{3p_k}{\psi_k-y_k}$&      $\eta$ &     $\min_{k \in \mathcal{N}^-} (y_k-\psi_k)$ & $\# N$ \\   
			\midrule[1pt]
		    1.0e+00&        1.63643751e-02& 3.74767766e+00& 0.00000000e+00& -2.10533852e-02&    3\\
		    1.0e+01&         1.58450462e-02& 3.61352870e+00& 0.00000000e+00& -2.07865748e-02&    4\\
		    1.0e+02&         1.97319240e-03& 1.04106448e-01& 0.00000000e+00& -8.25007721e-03&    8\\
		    1.0e+03&         -1.25202529e-02& -3.51210686e-01& -3.51210686e-01& -9.83097618e-04&   11\\
		    1.0e+04&         -2.42618027e-01& -3.37078927e-01& -3.37078927e-01& -9.97061633e-05&   12\\
		    1.0e+05&        -1.36127707e-01& -1.12485742e-01& -1.36127707e-01& -1.02473476e-05&   13\\
		    1.0e+06&         -9.75060053e-02& -3.06984382e-02& -9.75060053e-02& -1.03346212e-06&   15\\
		    1.0e+07&         -8.26915590e-02& -3.18139489e-02& -8.26915590e-02& -1.03449193e-07&   13\\
		    1.0e+08&         -1.81555278e-01& -3.47707710e-02& -1.81555278e-01& -1.03459974e-08&   14\\
		    1.0e+09&         -2.20032200e-01& -4.44457854e-02& -2.20032200e-01& -1.03461058e-09&   15\\
		    1.0e+10&         -2.24336956e-01& -4.55306972e-02& -2.24336956e-01& -1.03461166e-10&   17\\
		    1.0e+11&         -2.24772813e-01& -4.56405644e-02& -2.24772813e-01& -1.03461177e-11&   16\\
			\bottomrule
		\end{tabular}
	\end{table}
	
	\begin{figure}
		\begin{subfigure}[b]{0.5\textwidth}
			\includegraphics[trim = 70mm 85mm 75mm 70mm, width=0.5\textwidth]{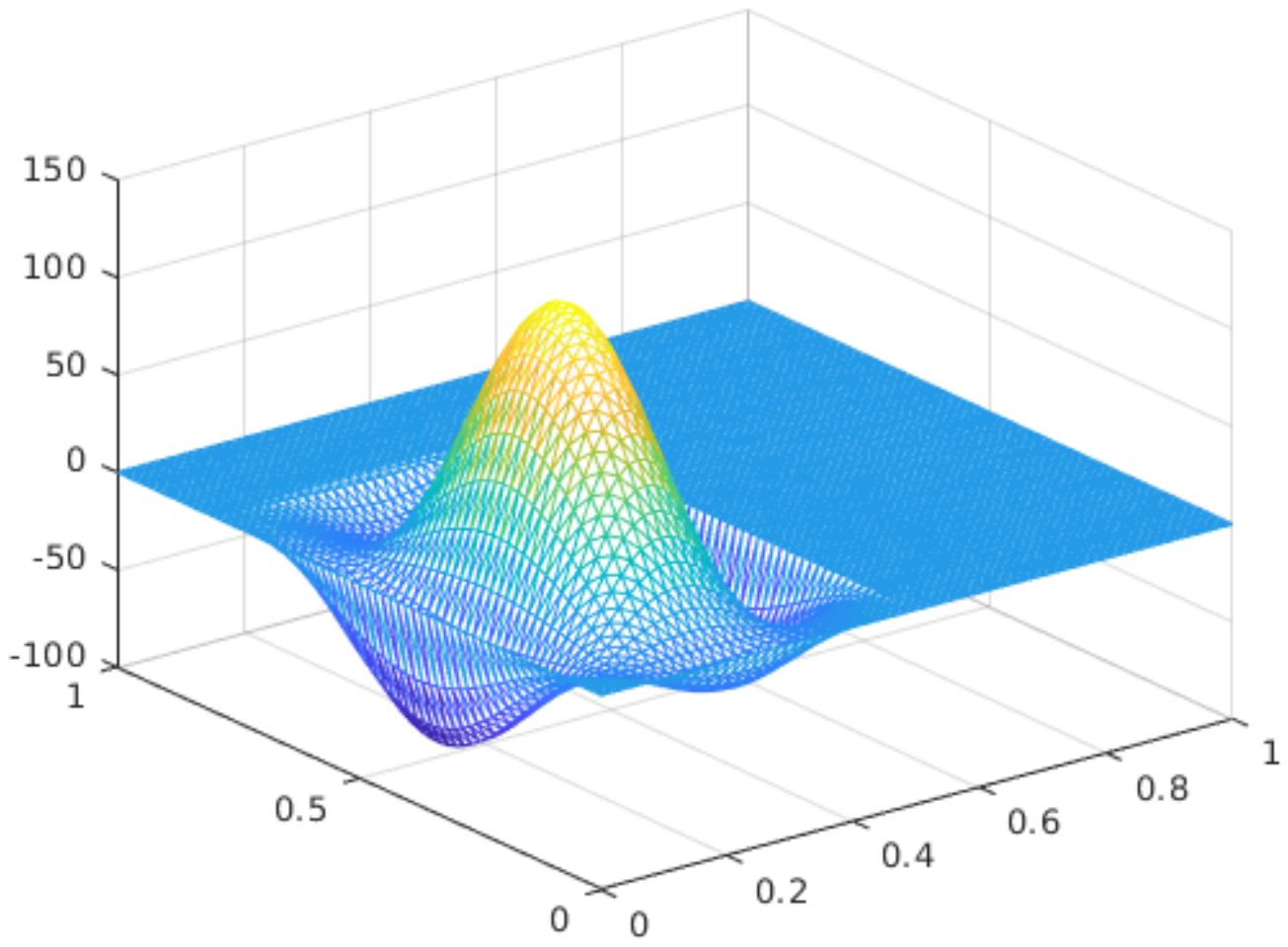}
			\caption{The optimal control}
		\end{subfigure}~ 
		\begin{subfigure}[b]{0.5\textwidth}
			\includegraphics[trim = 70mm 85mm 75mm 70mm, width=0.5\textwidth]{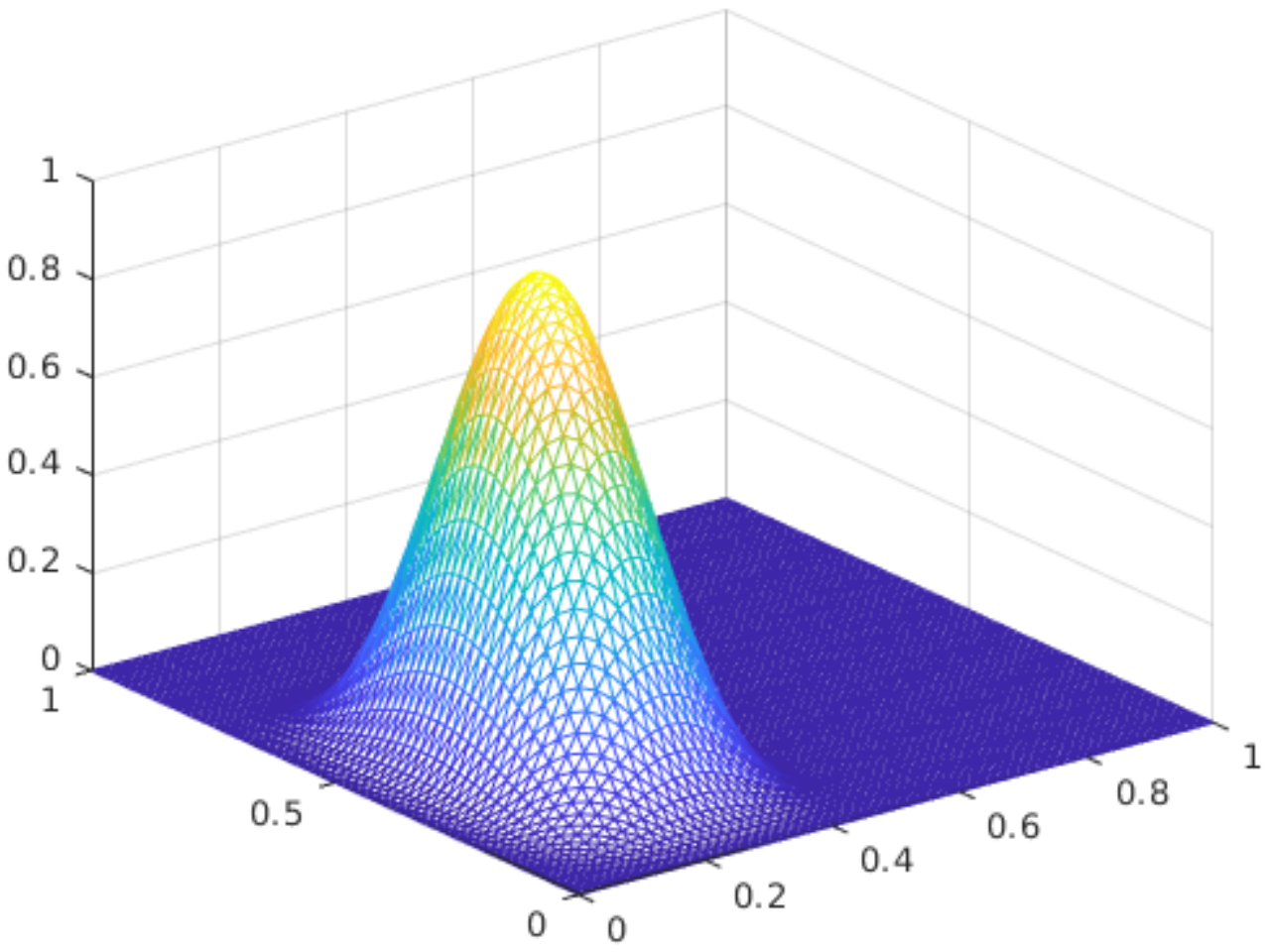}
			\caption{The optimal state}
		\end{subfigure}
		\begin{subfigure}[b]{0.5\textwidth}
			\includegraphics[trim = 70mm 85mm 75mm 70mm, width=0.5\textwidth]{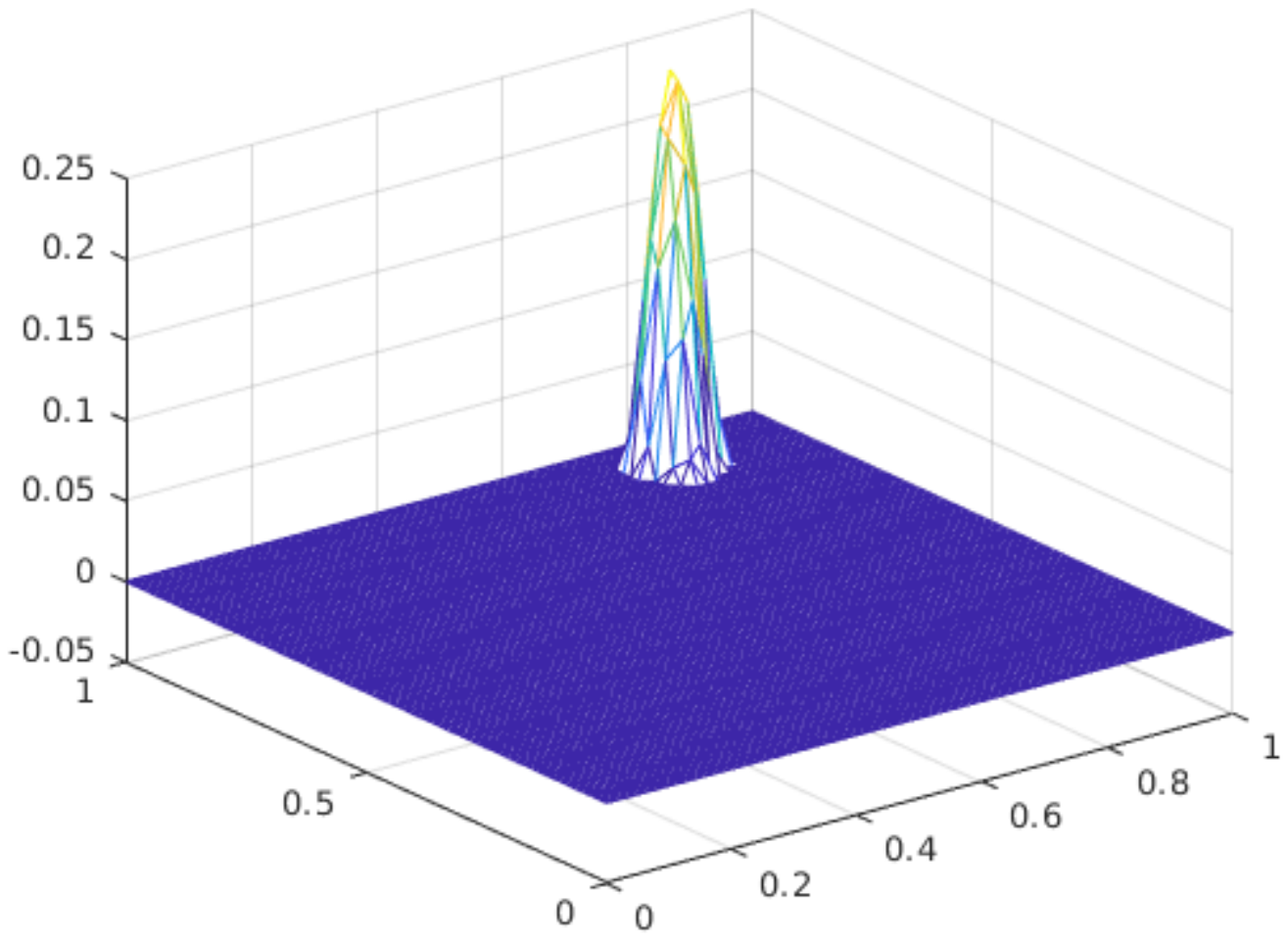}
			\caption{The optimal adjoint state}
		\end{subfigure}~
		\begin{subfigure}[b]{0.5\textwidth}
			\includegraphics[trim = 70mm 85mm 75mm 70mm, width=0.5\textwidth]{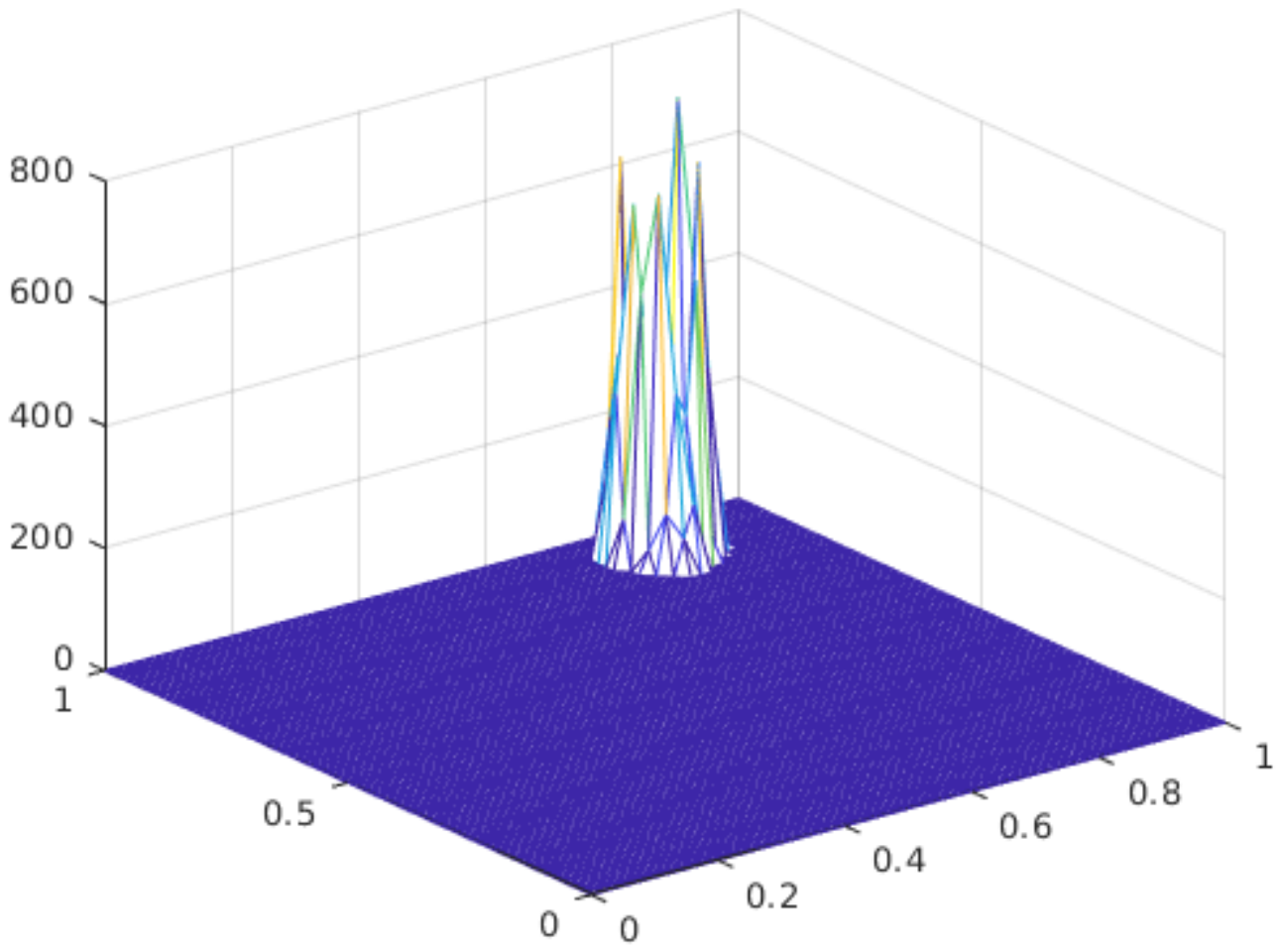}
			\caption{The multiplier $\mu$}
		\end{subfigure}
		
		\centering
		\begin{subfigure}[b]{0.5\textwidth}
			\includegraphics[trim = 70mm 85mm 75mm 70mm, width=0.5\textwidth]{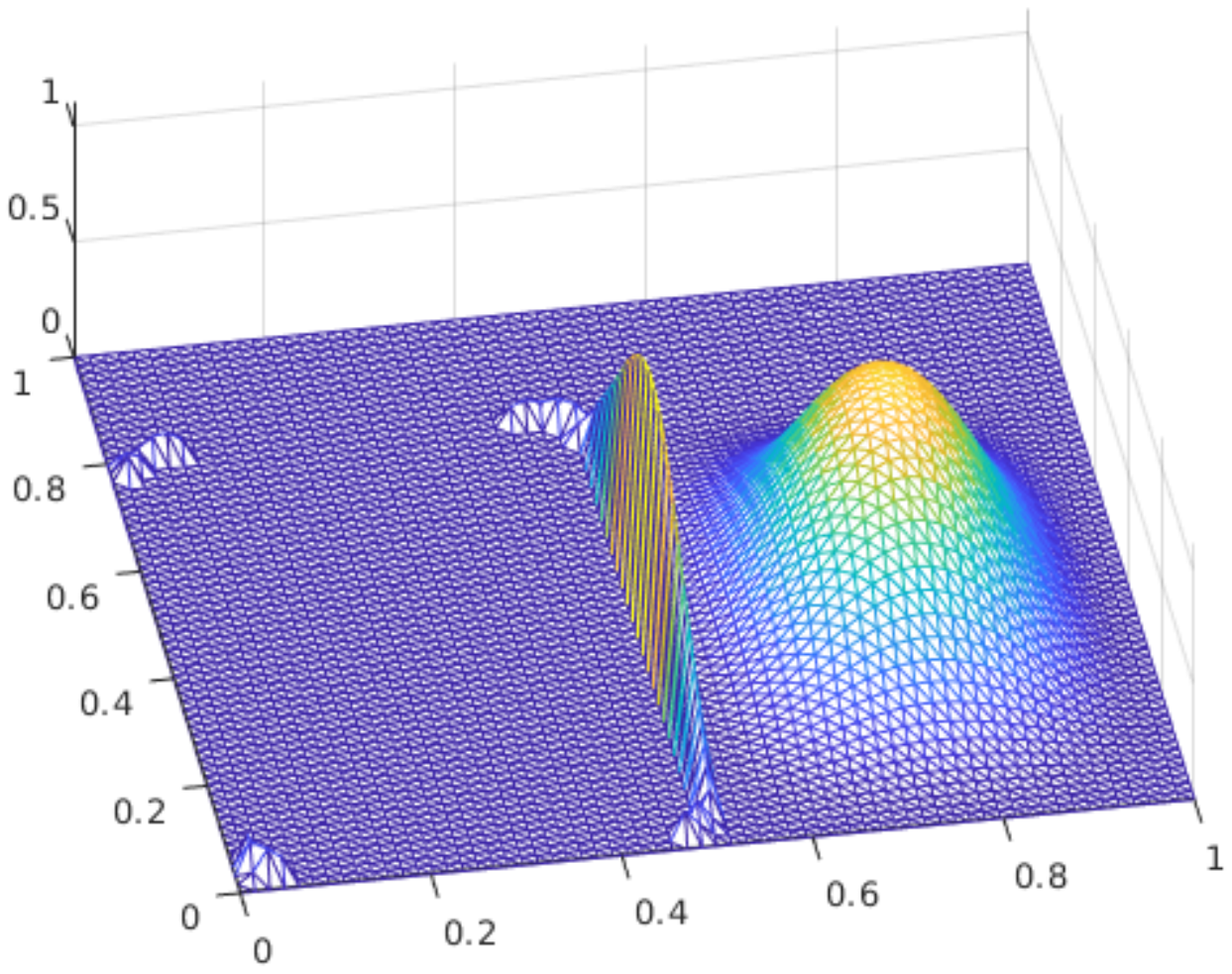}
			\caption{The multiplier $\xi$}
		\end{subfigure}~
		\caption{Example~\ref{example:mt1}: The optimal control, state, adjoint state, and the multipliers for $\gamma=10^8$.}\label{figure:mt1}
	\end{figure}
\end{example} 

\begin{example}
	\label{example:mrw2}
	The data for this problem is taken from Example~6.2 in \cite{MRW15} and for $(\mathbb{P})$ reads   
	\begin{align*}
	& \alpha=1, \quad \psi(x)=0 \mbox{ in } \Omega, \quad f(x)=\frac{1}{2}+\frac{1}{2}(x_1-x_2)\mbox{ in } \Omega,\\
	& y_0(x)=\begin{cases}
	-1       \quad \text{ if } |x|\geq 0.1,\\
	1-100x_1^2-50x_2^2  \quad \text{otherwise}.
	\end{cases}
	\end{align*}
	We have 
	\[
	\alpha \lambda_1+\sqrt{\alpha^2 \lambda^2_1+\alpha} \approx 19.3313.
	\]
	The numerical results are reported in Table~\ref{table:mrw2}. The condition  (\ref{optcondgammah}) is satisfied, and hence, the unique global solution has been computed, which we illustrate in Figure~\ref{figure:mrw2} for $\gamma=10^8$. 
	\begin{table}[h]
		\caption{Example~\ref{example:mrw2} where $\alpha \lambda_1+\sqrt{\alpha^2 \lambda^2_1+\alpha} \approx 19.3313$.}
		\label{table:mrw2}
		\begin{tabular}{ l    c  c c c  c}
			\toprule
			 $\gamma$ &         $\min_{k \in \mathcal{N}^+} \frac{p_k}{y_k-\psi_k}$ &  $\min_{k \in \mathcal{N}^-} \frac{3p_k}{\psi_k-y_k}$&      $\eta$ &     $\min_{k \in \mathcal{N}^-} (y_k-\psi_k)$ & $\# N$ \\   
			\midrule[1pt]
		    1.0e+00&         1.28191544e+00& 9.12899685e+00& 0.00000000e+00& -7.61042379e-03&    3\\
		    1.0e+01&         1.28189321e+00& 9.11848759e+00& 0.00000000e+00& -7.59555549e-03&    3\\
		    1.0e+02&         1.27854041e+00& 7.14870263e+00& 0.00000000e+00& -4.77178798e-03&    7\\
		    1.0e+03&         1.27598179e+00& 2.95273878e+00& 0.00000000e+00& -7.05858478e-04&   11\\
		    1.0e+04&         1.27568823e+00& 2.30861182e+00& 0.00000000e+00& -7.60379296e-05&   12\\
		    1.0e+05&         1.27561164e+00& 2.14128368e+00& 0.00000000e+00& -7.76063518e-06&   12\\
		    1.0e+06&         1.27562640e+00& 2.08304474e+00& 0.00000000e+00& -7.81857173e-07&   13\\
		    1.0e+07&         1.27563830e+00& 2.06636877e+00& 0.00000000e+00& -7.84995677e-08&   15\\
		    1.0e+08&         1.27563972e+00& 2.06470139e+00& 0.00000000e+00& -7.85310208e-09&   11\\
		    1.0e+09&         1.27563987e+00& 2.06453466e+00& 0.00000000e+00& -7.85341667e-10&   11\\
		    1.0e+10&         1.27563988e+00& 2.06451798e+00& 0.00000000e+00& -7.85344814e-11&   10\\
		    1.0e+11&         1.27563988e+00& 2.06451631e+00& 0.00000000e+00& -7.85345128e-12&   10\\
		    1.0e+12&         1.27563989e+00& 2.06451615e+00& 0.00000000e+00& -7.85345160e-13&   10\\
		    \bottomrule
		\end{tabular}
	\end{table}
	
	\begin{figure}
		\begin{subfigure}[b]{0.5\textwidth}
			\includegraphics[trim = 70mm 85mm 75mm 70mm, width=0.5\textwidth]{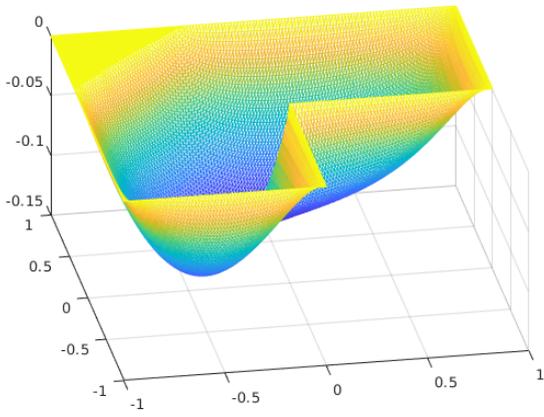}
			\caption{The optimal control}
		\end{subfigure}~ 
		\begin{subfigure}[b]{0.5\textwidth}
			\includegraphics[trim = 70mm 85mm 75mm 70mm, width=0.5\textwidth]{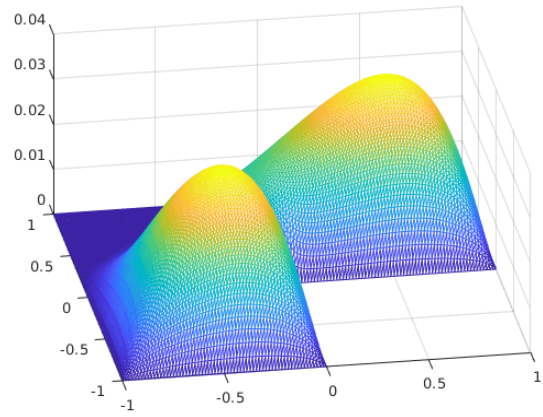}
			\caption{The optimal state}
		\end{subfigure}
		\begin{subfigure}[b]{0.5\textwidth}
			\includegraphics[trim = 70mm 85mm 75mm 70mm, width=0.5\textwidth]{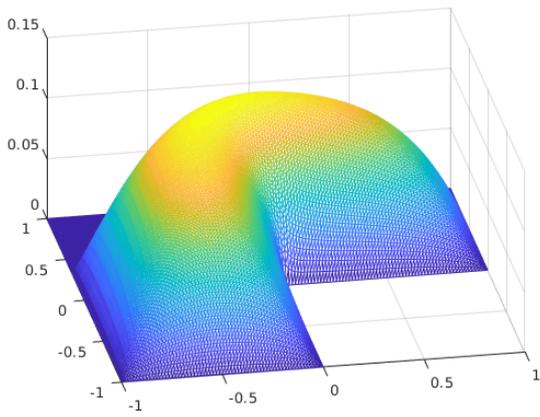}
			\caption{The optimal adjoint state}
		\end{subfigure}~
		\begin{subfigure}[b]{0.5\textwidth}
			\includegraphics[trim = 70mm 85mm 75mm 70mm, width=0.5\textwidth]{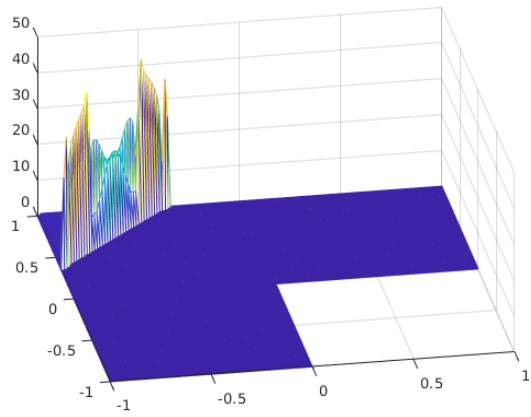}
			\caption{The multiplier $\mu$}
		\end{subfigure}
		
		\centering
		\begin{subfigure}[b]{0.5\textwidth}
			\includegraphics[trim = 70mm 85mm 75mm 70mm, width=0.5\textwidth]{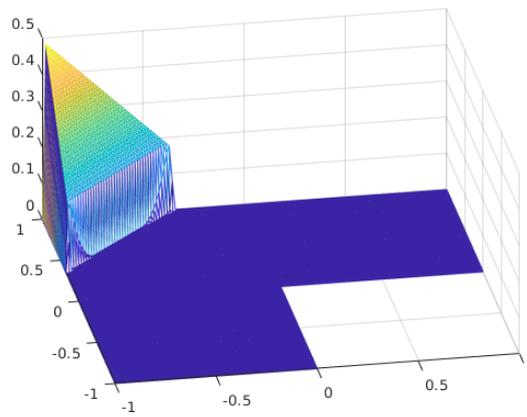}
			\caption{The multiplier $\xi$}
		\end{subfigure}
		\caption{Example~\ref{example:mrw2}: The optimal control, state, adjoint state, and the multipliers for $\gamma=10^8$.}\label{figure:mrw2}
	\end{figure}

\end{example}

\newpage

\bibliographystyle{plain}
\bibliography{references}

\begin{thebibliography}{99}
\bibitem{Hinze1}
A. Ahmad Ali, K. Deckelnick and M. Hinze, {\it Global minima for semilinear optimal control problems}. Comput. Optim. Appl. \textbf{65}, 261--288 (2016).
\bibitem{Hinze2} A. Ahmad Ali, K. Deckelnick and M. Hinze, {\it Error analysis for global minima of semilinear optimal control problems}. Mathematical Control and related Fields \textbf{8}, 195--215 (2018).
\bibitem{Ba84} V. Barbu, {\it Optimal control of variational inequalities}. Res. Notes Math. {\bf 100}, Pitman, Boston, 1984.
\bibitem{B15} S. Bartels, {\it Numerical methods for nonlinear partial differential equations}. Springer Series in Computational Mathematics, {\bf 47}. Springer, Cham, 2015.
\bibitem{BM00} M. Bergounioux and F. Mignot, {\it Optimal control of obstacle problems: existence of Lagrange multipliers}.  ESAIM Control Optim. Calc. Var. {\bf 5}, 45--70 (2000).
\bibitem{H01} M. Hinterm\"uller, {\it Inverse coefficient problems for variational inequalities: Optimality conditions and numerical realization}. 
ESAIM Math. Model. Numer. Anal. {\bf 35}, 129--152 (2001). 
\bibitem{HK11} M. Hinterm\"uller and I. Kopacka, {\it A smooth penalty approach and a nonlinear multigrid algorithm for elliptic MPECs}. Comput. Optim. Appl. {\bf 50}, 111--145 (2011).
\bibitem{IK00}  K. Ito and K. Kunisch, {\it Optimal control of elliptic variational inequalities}. Appl. Math. Optim. {\bf 41}, 343--364 (2000).
\bibitem{KW12a} K. Kunisch and D. Wachsmuth, {\it Path--following for optimal control of stationary variational inequalities}. Comput. Opt. Appl. {\bf 51}, 1345--1373 (2012).
\bibitem{KW12b} K. Kunisch and D. Wachsmuth, {\it Sufficient optimality conditions and semi--smooth Newton methods for optimal control of stationary variational inequalities}. 
 ESAIM Control Optim. Calc. Var. {\bf 18}, no. 2, 520–-547 (2012).
\bibitem{MT13}  C. Meyer and O. Thoma, {\it A priori finite element error analysis for optimal control of the obstacle problem}. SIAM J. Numer. Anal. {\bf 51}, 605--628 (2013).
\bibitem{MRW15}  C. Meyer, A. Rademacher, and W. Wollner, {\it Adaptive optimal control of the obstacle problem}. SIAM J. Sci. Comput. {\bf 37}, 918--945 (2015).
\bibitem{M76} F. Mignot, {\it Contr\^{o}le dans les in\'{e}quations variationelles  elliptiques}. J. Funct. Anal. {\bf 22}, 130--185 (1976).
\bibitem{MP84} F. Mignot and J.-P. Puel, {\it Optimal control in some variational inequalities}. SIAM J. Control Optim. {\bf 22} (3), 466--476 (1984).
\bibitem{SW13} A. Schiela and D. Wachsmuth, {\it Convergence analysis of smoothing methods  for optimal control of stationary variational inequalities  with control constraints}.
 ESAIM Math. Model. Numer. Anal. {\bf 47}, no. 3, 771--787 (2013). 
\bibitem{W14} G. Wachsmuth, {\it Strong stationarity for optimal control of the obstacle problem with control constraints}. SIAM J. Optim. {\bf 24} (1), 1914--1932 (2014).
\end{thebibliography}
\end{document}